\documentclass[]{article}

\usepackage{amsmath,amssymb,amsthm}
\usepackage{hyperref}
\usepackage{tikz}
\usepackage{algorithm}
\usepackage{algpseudocode}

\usetikzlibrary{intersections}

\newcommand{\Z}{\mathbb{Z}}
\newcommand{\N}{\mathbb{N}}

\newcommand{\signal}[1]{#1_{\mathrm{sig}}}
\newcommand{\hsignal}[1]{#1_{\mathrm{hsig}}}
\newcommand{\vsignal}[1]{#1_{\mathrm{vsig}}}

\newcommand{\north}[1]{#1_{\mathrm{north}}}
\newcommand{\south}[1]{#1_{\mathrm{south}}}
\newcommand{\east}[1]{#1_{\mathrm{east}}}
\newcommand{\west}[1]{#1_{\mathrm{west}}}

\theoremstyle{plain}
\newtheorem{lemma}{Lemma}
\newtheorem{definition}[lemma]{Definition}
\newtheorem{corollary}[lemma]{Corollary}
\newtheorem{proposition}[lemma]{Proposition}
\newtheorem{theorem}[lemma]{Theorem}

\newtheorem{question}[lemma]{Question}

\theoremstyle{example}
\newtheorem{example}[lemma]{Example}

\title{Countable Sofic Shifts with a Periodic Direction}
\author{Ilkka T\"orm\"a\footnote{Reseasch supported by Academy of Finland grant 295095.} \\ University of Turku, Finland}

\begin{document}

\maketitle

\begin{abstract}
As a variant of the equal entropy cover problem, we ask whether all multidimensional sofic shifts with countably many configurations have SFT covers with countably many configurations.
We answer this question in the negative by presenting explicit counterexamples.
We formulate necessary conditions for a vertically periodic shift space to have a countable SFT cover, and prove that they are sufficient in a natural (but quite restricted) subclass of shift spaces.
\end{abstract}

\section{Introduction}

In the context of one-dimensional symbolic dynamics, the structural properties and connections between shifts of finite type (SFTs for short) and sofic shifts are quite well understood.
They can be represented as sets of infinite walks on finite directed graphs, with labeled edges or vertices in the case of sofic shifts.
Furthermore, to each sofic shift one can associate several SFT covers that are in some sense minimal: Fischer covers, Krieger covers, past set covers et cetera.
In particular, these covers have the same topological entropy as the corresponding sofic shifts, and are topologically transitive if and only if the sofic shifts are.
Several subclasses, such as shifts of almost finite type, can be defined via relations between a sofic shift and its set of SFT covers.
See \cite{LiMa95} for an overview of this topic.

In multidimensional symbolic dynamics, this connection much weaker.
While each sofic shift has SFT covers whose entropy is arbitrarily close to their own \cite{De06}, it is not known whether equal entropy covers always exist.
The class of sofic shifts is likewise more nebulous.
Several sufficient and necessary conditions for soficity of a multidimensional shift space are known, but between them remains a rather large gap.
These conditions are both computational \cite{DuLeSh08,Ho09,DuRoSh12,AuSa13} and geometric \cite{KaMa13,Pa13,OrPa16}, and recently a new class of non-sofic shift spaces was announced that fools all previously known soficness tests \cite{DeRo18}.

In this article, we concentrate on the subclass of countable shift spaces, that is, shift spaces with countably many configurations.
The research of this particular class was initiated in \cite{BaDuJe08} in the context of tiling systems.
In one dimension, countable sofic shifts are exactly those with zero topological entropy, but in the multidimensional case not all zero entropy sofic shifts are countable.
Our main result settles an analogue of the equal entropy cover problem of multidimensional sofic shifts: whether every countable sofic shift has a countable SFT cover.
The answer is negative, and our surprisingly succinct topological proof presents an explicit counterexample.
The counterexample is vertically constant, meaning that each of its configurations has $(0,1)$ as a vector of periodicity, and this periodicity is an essential part of the proof.
It readily generalizes to an arbitrary number of dimensions.
We also formulate three necessary conditions, one geometric and two computational, for possessing a countable SFT cover, and prove that they are sufficient for the class of shift spaces whose horizontal rows come from a one-dimensional countable sofic shift.
It is left as an open problem whether the conditions are sufficient in general.

\section{Definitions}

\subsection{Symbolic Dynamics}

Let $\Sigma$ be a finite alphabet and $d \geq 1$.
The set $\Sigma^{\Z^d}$ equipped with the product topology is the \emph{$d$-dimensional full shift}, and its elements are called \emph{configurations}.
We will mostly concentrate on the cases $d = 1$ and $d = 2$.
The group $\Z^d$ acts on $\Sigma^{\Z^d}$ by shifting: $\sigma^{\vec v}(x)_{\vec w} = x_{\vec v + \vec w}$.

A \emph{shift space} is a topologically closed subset of $\Sigma^{\Z^d}$ which is invariant under the shift action.
Every shift space is defined by a set of \emph{forbidden finite patterns} as the set of configurations where none of these patterns occur.
If the set of forbidden patterns is finite, the corresponding shift space is a \emph{shift of finite type}, or SFT for short.
If the set of forbidden patterns (equivalently, all patterns not occurring in the shift) is computably enumerable, the shift space is \emph{effectively closed}.
In our constructions, $\Sigma$ will usually be a set of decorated square tiles, and the forbidden patterns are those where the decorations of neighboring tiles do not match.
Sometimes we will impose additional local constraints.

A configuration $x \in \Sigma^{\Z^d}$ is \emph{periodic} in a direction $\vec v \in \Z^d$, or that $\vec v$ is a \emph{period} of $x$, if $\sigma^{\vec v}(x) = x$.
We say $x$ is \emph{vertically constant}, if the basis vectors $\vec e_2, \ldots, \vec e_d$ are periods of $x$, and \emph{vertically periodic}, if some of their multiples are periods of $x$.
For dimensions greater than 2, we thus view the first coordinate as `horizontal' and the rest as `vertical'.
A one-dimensional configuration $y \in \Sigma^\Z$ is \emph{ultimately periodic to the right}, if there exists $p > 0$ such that $x_n = x_{n+p}$ for all large enough $n \in \Z$.
Ultimate periodicity to the left is defined analogously.
To help with the construction of vertically constant sofic shifts, we use the following powerful result.

\begin{lemma}[\cite{DuRoSh12,AuSa13}]
  \label{lem:BlackBox}
  For $d \geq 2$, every effectively closed vertically constant $\Z^d$-shift space is sofic.
\end{lemma}

The orbit closure of $x \in \Sigma^{\Z^d}$ under a single shift map $\sigma^{\vec v}$ is denoted $\mathcal{O}_{\sigma^{\vec v}}(x)$.
A configuration $x \in \Sigma^{\Z^2}$ is \emph{uniformly recurrent in the direction $\vec v \in \Z^2$}, if the orbit closure $\mathcal{O}_{\sigma^{\vec v}}(x)$ is a minimal system under $\sigma^{\vec v}$.
This is equivalent to the following condition: for each finite pattern $p$ that occurs in $x$ at some $\vec w \in \Z^2$, there exists $K \in \N$ such that for each $n \in \Z$, there exists $m \in [n-K, n+K]$ such that $p$ occurs in $x$ at $\vec w + m \vec v$.
In particular, if $x$ is $\vec v$-periodic, then it is uniformly recurrent in the direction $\vec v$.

For $e \leq d$, the $e$-dimensional \emph{projective subdynamics} of a shift space $X \subset \Sigma^{\Z^d}$ is denoted $P_e(X) \subset \Sigma^\Z$, and it consists of the restriction to $\Z^e$ of each configuration of $X$; we choose $\Z^e$ as the subgroup of $\Z^d$ where the last $d-e$ coordinates have value $0$.
We denote $P_1(X) = P(X)$.
In the two-dimensional case, $P(X)$ consists of the central horizontal row of each configuration of $X$.
If $Y \subset \Sigma^{\Z^e}$ is a shift space, then $P_e^\dag(Y) \subset \Sigma^{\Z^d}$ is the vertically constant shift space with $P_e(P_e^\dag(Y)) = Y$.

A \emph{block code} is a continuous map $\psi : X \to Y$ between two $\Z^d$-shift spaces that intertwines the shift actions.
It is defined by a finite \emph{neighborhood} $N \subset \Z^d$ and a \emph{local rule} $\Psi : \Sigma^N \to \Sigma$, in the sense that $\psi(x)_{\vec v} = \Psi(\sigma^{\vec v}(x)|_N)$ holds for all $x \in X$ and $\vec v \in \Z^d$.
The image of an SFT under a block code is a \emph{sofic shift}, and the SFT and block code form its \emph{SFT cover}.

An \emph{edge shift} is a one-dimensional SFT $X_G$ defined by a finite directed graph $G = (V, E)$.
The configurations of the edge shift are the edges of all bi-infinite paths in the graph, so it is a subset of $E^\Z$.
An edge labeling $\lambda : E \to \Gamma$ induces a block map from $X_G$ to $\Gamma^\Z$ by acting on each edge, and $\lambda$ is \emph{right-resolving} if for each vertex $v \in V$, $\lambda$ is injective on the outgoing edges of $v$.
Each one-dimensional sofic shift $Y$ has an SFT cover consisting of a right-resolving labeling of an edge shift.
We make use of the following result about the structure of one-dimensional countable sofic shifts.

\begin{lemma}[Lemma~4.8 in~\cite{PaSc15}]
  \label{lem:CountableStructure}
  Let $X \subset \Gamma^\Z$ be a one-dimensional countable sofic shift.
  Then there exists an edge shift $Y$ given by a graph $G = (V, E)$ and a right-resolving factor map $\psi : Y \to X$ such that $G$ is a disjoint union of components $K_1, \ldots, K_n$, where each $K_i$ consists of a finite number of vertex-disjoint cycles $c_{i,0}, \ldots, c_{i, m_i}$ with one marked vertex, and vertex-disjoint nonempty paths $p_{i,1}, \ldots, p_{i,m_i}$, where each $p_{i,j}$ connects the marked vertex of $c_{i,j-1}$ to the marked vertex of $c_{i,j}$.
\end{lemma}

\subsection{Counter Machines}
\label{sec:Machines}

In this article, we use several variants of counter machines to implement computation.
We only describe their properties informally.
For $k \in \N$, a deterministic $k$-counter machine $M = (k, Q, \delta, q_0, q_f)$ consists of a finite set $Q$ of internal states and $k$ counters, each of which holds a natural number.
A configuration of the machine is thus an element of $Q \times \N^k$.
Some of the counters may be \emph{input counters}, which means that their values are fixed at the beginning of the computation and cannot be modified; the others are \emph{internal counters}.
The machine is initialized in a state $q_0 \in Q$ with the internal counters set to $0$, and on each step of its computation, it may compare any internal counter for equality to $0$ or an input counter, and based on that information and its internal state, it may increment or decrement one or more internal counters and update its internal state.
The element $\delta$ of the definition encodes the transition function of $M$ in some suitable way.

An \emph{integer $k$-counter machine} is similar to a $k$-counter machine, except that the values of the counters can be negative.
An integer counter machine can have an \emph{oracle counter}, which works as follows.
There is a separate oracle alphabet $\Sigma$, and an oracle configuration $x \in \Sigma^\Z$.
During the computation, the machine can modify the oracle counter normally, and it can query the symbol $x_n$, where $n \in \Z$ is the current value of the oracle counter.

Whan a counter machine of any kind enters the final state $q_f$, it halts.
What this means in the context of a machine simulated by a tiling depends on the application: it may result in a tiling error or produce some special tile.

Both types of counter machines can be simulated by countable SFTs.
We adapt the methods of \cite{SaTo13pub} for this.
Consider first a $k$-counter machine $M = (k, Q, \delta, q_0, q_f)$.
The SFT $X_M$ that simulates $M$ has as its alphabet $S_M = \{\#_L, \#_R\} \cup \left( (\{\mathrm{L}, \mathrm{C}, \mathrm{R}\} \cup H_M) \times \prod_{i=1}^k \{\mathrm{P}_i, \mathrm{Z}_i\} \right)$, where $H_M$ is an auxiliary \emph{head alphabet}.
  The non-$\#$ components of $S_M$ are called the \emph{head track} and \emph{counter tracks}.
  The tracks of each horizontal row of $X_M$ must have the form $\mathrm{L}^m h \mathrm{C}^n \mathrm{R}^k$ for the head track or $\mathrm{P}_i^{m_i} \mathrm{Z}_i^{n_i}$ for the counter tracks for $m, n, k, m_i, n_i \in \N \cup \{\infty\}$ and $h \in H_M$, or a degenerate version of such configuration where some of the regions don't exist.
  The tracks can be followed by regions of by $\#_L$ on the left and $\#_R$ on the right.
  The number $m_i$ represents the value of counter $i$.
  The symbol $h$ is the \emph{zig-zag head}, which contains the internal state of $M$ and some extra information that guides the computation.
  The region containing the tracks is the \emph{computation cone}, and if its length is nonzero, then it grows by one cell to the right on the next row.
  If the computaion cone is empty (meaning that the computation has not sparted yet), then the next row can be equal to the current one, or it can grow by one cell to contain the zig-zag head in the initial state $q_0$.
  The initial values of read-only counters are unrestricted and are `hidden' until the computation cone extends sufficiently far to reveal them, but the internal counters are forced to have value $0$ at the beginning of the computation.
  The zig-zag head travels back and forth in the computation cone at a speed of two cells per row, updating its internal state at the left border where it can see which counters are positive, and updating the internal counter values as it passes over them.
  As the head passes over an internal counter, it also remembers which read-only counters hold the same value.
  The counter machine can introduce tiling errors by entering an internal state that has no outgoing transition in $\delta$.
  It can be shown that the resulting SFT $X_M$ is countable.
  
  See Figure~\ref{fig:CounterMachine} for a sample simulation of a machine with two counters.
  The computation starts with a read-only (lower) counter initialized at value $5$, and the internal (upper) counter at $0$.
  On the first two computation steps, the machine increments the internal counter.
  Note that the machine may compare the interal counter to the input counter already in the first step, even though the value of the latter is hidden; since internal counters are never hidden, they are known to be unequal to any hidden counter.
  
  The simulation of an integer $k$-counter machine is similar, but there are three major differences.
  First, the computation cone grows both to the left and to the right, at the rate of one step per row.
  Second, there is an extra counter that is always initialized at $0$ and never changes its value.
  It denotes the value $0$, and other counters can be compared to it locally.
  Third, there is a separate vertically constant layer over the oracle alphabet $\Sigma$, and when the zig-zag head travels over the oracle counter, it can read the symbol under it on this new layer.
  If this layer is unrestricted, then the resulting SFT will be uncountable, so in our construction we impose additional constraints on it to ensure countability.
  
  \begin{figure}[htp]
  \begin{center}
  \begin{tikzpicture}[scale=0.6]
  
  
    \foreach \y in {-1,...,13}{
      \node at (0.5,\y+0.5) {$\#_L$};
      \node at (1.5,\y+0.5) {$\#_L$};
    }
    


    \foreach \x in {2,...,11}{
      \node at (\x+0.5,-0.5) {$\#_R$};
    }
    
    \foreach \y in {0,...,9}{
      \pgfmathsetmacro{\xmin}{\y+2}
      \foreach \x in {\xmin,...,11}{
        \node at (\x+0.5,\y+0.5) {$\#_R$};
      }
    }
    
    
    \foreach \x/\d [count = \y] in {2/\rightarrow,
                                    2/\leftarrow,
                                    2/\rightarrow,
                                    4/\rightarrow,
                                    6/\leftarrow,
                                    4/\leftarrow,
                                    2/\rightarrow,
                                    4/\rightarrow,
                                    6/\rightarrow,
                                    8/\rightarrow,
                                    10/\rightarrow
                                    }{
      \node at (\x+0.3,\y+0.3) {$q$};
      \node at (\x+0.3,\y+0.6) {${}^\d$};
    }
    
    
    \node at (2.7,1.7) {\small Z};
    
    \foreach \y in {2,...,7}{
      \node at (2.7,\y+0.7) {\small P};
      \node at (3.7,\y+0.7) {\small Z};
    }
    \foreach \y in {8,...,13}{
      \node at (2.7,\y+0.7) {\small P};
      \node at (3.7,\y+0.7) {\small P};
    }
    \foreach \y [count = \c] in {3,...,9}{
      \foreach \x in {1,...,\c}{
        \node at (\x+3.7,\y+0.7) {\small Z};
      }
    }
    \foreach \y in {10,...,13}{
      \foreach \x in {4,...,11}{
        \node at (\x+0.7,\y+0.7) {\small Z};
      }
    }
    
    
    \foreach \y in {1,...,13}{
      \pgfmathsetmacro{\xmax}{min(\y+1,6)}
      \foreach \x in {2,...,\xmax}{
        \node at (\x+0.7,\y+0.3) {\small P};
      }
    }
    \foreach \y in {6,...,13}{
      \pgfmathsetmacro{\xmax}{min(\y+1,11)}
      \foreach \x in {7,...,\xmax}{
        \node at (\x+0.7,\y+0.3) {\small Z};
      }
    }
    
    
%
    
    \draw (0,-1) grid (12,14);
  
  \end{tikzpicture}
  \end{center}
  \caption{A configuration of $X_M$. The position and direction of the zig-zag and two counter tracks are shown, other states are not.}
  \label{fig:CounterMachine}
  \end{figure}

\section{Countable Cover Conditions}

In this section we explore some necessary conditions for possessing a countable SFT cover, and provide examples of shift spaces that do or do not satisfy them.
The following lemma, whose proof is topological, is the main workhorse behind our results.

\begin{lemma}
  \label{lem:UniformlyRecurrent}
  Let $X \subset \Sigma^{\Z^d}$ and $Y \subset \Gamma^{\Z^d}$ be shift spaces, and let $\phi : X \to Y$ be a factor map.
  Let $y \in Y$ be uniformly recurrent in a direction $\vec v \in \Z^d$.
  Then there exists $x \in \phi^{-1}(y)$ that is uniformly recurrent in the direction $\vec v$.
  If $X$ is countable, then $x$ is periodic in the direction $\vec v$.
\end{lemma}

\begin{proof}
  Denote $Z = \phi^{-1}(\mathcal{O}_{\sigma^{\vec v}}(y)) \subset X$, which is a compact set.
  We have $\sigma^{\vec v}(Z) = Z$, and hence $(Z, \sigma^{\vec v})$ is a topological dynamical system.
  Let $U \subset Z$ be a minimal subsystem of $Z$, which exists by a simple application of Zorn's lemma.
  Then any $z \in U$ is uniformly recurrent in the direction $\vec v$.
  Since $\phi(U)$ is a subsystem of the orbit closure $(\mathcal{O}_{\sigma^{\vec v}}(y), \sigma^{\vec v})$, which is minimal, we have $\phi(U) = \mathcal{O}_{\sigma^{\vec v}}(y)$ and hence $y \in \phi(U)$.
  Any element of $U \cap \phi^{-1}(y)$ can be chosen as $x$.
  
  If the shift space $X$ is countable, then so is $U$.
  Since an infinite minimal dynamical system is uncountable, $U$ must be finite.
  Then each $x \in U$ satisfies $\sigma^{k \vec v}(x) = x$ for some $k \geq 1$.
\end{proof}

\begin{lemma}
  \label{lem:UltimatelyPeriodic}
  Let $X \subset \Sigma^{\Z^d}$ be a countable SFT and $Y \subset \Gamma^{\Z^d}$ a vertically constant shift space, and let $\phi : X \to Y$ be a factor map.
  Then each $y \in Y$ has a $\phi$-preimage $x \in X$ which is vertically periodic, and each horizontal row of $x$ is ultimately periodic in both directions.
\end{lemma}

In particular, each horizontal row of each $y \in Y$ is ultimately periodic to the left and right.

\begin{proof}
  Let $p_2, \ldots, p_d \geq 1$ and $\vec p = (p_2, \ldots, p_d) \in \Z^{d-1}$, and consider the subset $Z_{\vec p} \subset X$ of configurations that are $p_i \vec e_i$-periodic for all $2 \leq i \leq d$.
  We can define $Z_{\vec p}$ by adding a finite number of forbidden patterns to those that define $X$, so it is an SFT.
  Denote $P = \{0\} \times \prod_{i=2}^d \{0, \ldots, p_i-1\} \subset \Z^d$.
  Define a map $\pi_{\vec p} : Z_{\vec p} \to (\Sigma^P)^\Z$ by $\pi_{\vec p}(z)_n = \sigma^n(z)|_P$ for each $n \in \N$.
  Then $\pi_{\vec p}$ is injective on $Z_{\vec p}$, and the image $\pi_{\vec p}(Z_{\vec p})$ is a one-dimensional shift space whose symbols are blocks of shape $P$ drawn from $X$.
  Since $X$ is countable, so is $\pi_{\vec p}(Z_{\vec p})$, and it is easily seen to be an SFT as well.
  Lemma~\ref{lem:CountableStructure} implies that each configuration of the countable SFT $\pi_{\vec p}(Z_{\vec p})$, and hence each row of $Z_{\vec p}$, is ultimately periodic to the left and right.
  
  Let $x \in X$ be a vertically periodic $\phi$-preimage of $y$, given by Lemma~\ref{lem:UniformlyRecurrent}, and let $\vec p \in \Z^{d-1}$ be the vector of its periods in the directions $\vec e_i$ as above.
  Then $x \in Z_{\vec p}$, and hence it is ultimately periodic to the left and right, and so is $y$.
\end{proof}

\begin{lemma}
  \label{lem:ComputabilityCondition}
  Let $Y \subset \Gamma^{\Z^d}$ be a countably covered vertically constant sofic shift.
  Given three nonempty words $u, v, w \in \Gamma^+$, it is decidable whether ${}^\infty u v w^\infty \in P(Y)$.
\end{lemma}

\begin{proof}
  Let $X \subset \Sigma^{\Z^d}$ be a countable SFT and $\phi : X \to Y$ a factor map.
  Let $y \in \Gamma^{\Z^d}$ be the configuration whose every horizontal row is equal to ${}^\infty u v w^\infty$.
  We have ${}^\infty u v w^\infty \in P(Y)$ if and only if $y \in Y$.
  
  Let $r \in \N$ be the radius of $\phi$.
  If $y \notin Y$, then by compactness there exists $m \in \N$ such that there exists no $(2 m + 2r + 1)^d$-shaped pattern $Q$ over $\Sigma$ which is locally valid with respect to the rules of $X$, and satisfies $\phi(Q) = y|_{[-m, m]^d}$.
  This $m$ can be found algorithmically by enumerating locally valid patterns.
  
  Suppose then that $y \in Y$.
  By Lemma~\ref{lem:UltimatelyPeriodic}, there exists $x \in \phi^{-1}(y)$ which is vertically periodic, and whose rows are ultimately periodic to the left and right.
  Such an $x$ can be found algorithmically by enumerating all possible vertical periods, eventual horizontal periods and transient parts, and verifying whether $x \in X$ and $\phi(x) = y$.
\end{proof}

Lemma~\ref{lem:UltimatelyPeriodic} and Lemma~\ref{lem:ComputabilityCondition} give two properties that all vertically constant countably covered sofic shifts satisfy, one of which is geometric and the other computational in nature.
All sofic shifts are effectively closed, which gives another computational necessary condition.
We formalize these results in a definition.

\begin{definition}
  Let $Y \subset \Gamma^\Z$ be a one-dimensional shift space.
  We say that $Y$ satisfies the \emph{countable cover conditions}, if
  \begin{enumerate}
    \item $Y$ is effectively closed,
    \item every $y \in Y$ is ultimately periodic to the left and right, and
    \item given $u, v, w \in \Gamma^+$, it is decidable whether ${}^\infty u v w^\infty \in Y$.
  \end{enumerate}
  A $d$-dimensional shift space $Y \subset \Gamma^{\Z^d}$ satisfies the countable cover conditions if it is vertically constant and $P(Y)$ satisfies the one-dimensional conditions.
\end{definition}

\begin{corollary}
  \label{cor:Main}
  For $d \geq 2$, every countably covered vertically constant $\Z^d$-sofic shift satisfies the countable cover conditions.
\end{corollary}

Since the last two conditions allow us to enumerate the configurations of the shift space, we can also enumerate its language.

\begin{proposition}
  Every shift space that satisfies the countable cover conditions has a computable language.
\end{proposition}

We now give examples of vertically constant sofic shifts that do not satisfy the countable cover conditions, showing that they are not countably covered.

\begin{example}
\label{ex:First}
  Let $\Gamma = \{0, 1\}$, and consider the one-dimensional shift space $Z \subset \Gamma^\Z$ defined by the forbidden patterns $0^m 1 0^m 1$ for all $m > 0$, as well as $1 0^n 1 1$ and $1 0^n 1 0^{n + 2}$ for all $n \geq 0$.
  It is the orbit closure of the configuration $z = {}^\infty 0 . 1 1 0 1 0 0 1 0 0 0 1 \cdots$, where the number of $0$s between consecutive $1$s always increases by one.
  Since $Z$ is effectively closed, Lemma~\ref{lem:BlackBox} implies that the $d$-dimensional vertically constant shift space $P^\dag(Z) \subset \Gamma^{\Z^d}$ is sofic for each $d \geq 2$.
  It is also countable, since $Z$ consists of the orbit of $z$, the orbit of ${}^\infty 0 1 0^\infty$, and the all-$0$ configuration.
  However, since the configuration $z \in Z$ is not ultimately periodic to the right, the shift space $Z$ does not satisfy the second countable cover condition, and Corollary~\ref{cor:Main} implies that $P^\dag(Z)$ is not countably covered.
\end{example}

\begin{example}
  Let $\Gamma = \{0, 1, 2, 3\}$, and consider the one-dimensional shift space $Z \subset \Gamma^\Z$ defined by the forbidden patterns $a b$ for $a > b \in \Gamma$, the pattern $0 1^m 2^n 3$ for each $m, n \in \N$ such that the $m$th Turing machine $T_m$ does not halt after exactly $n$ steps on empty input, as well as $0 1^m 2^{n+1}$ for each $m, n \in \N$ such that $T_m$ halts after $n$ steps.
  Then, for each $m \in \N$, the shift space $Z$ contains the configuration ${}^\infty 0 1^m 2^n 3^\infty$ if $T_m$ halts after exactly $n$ steps, and ${}^\infty 0 1^m 2^\infty$ if $T_m$ never halts.
  These are the only configurations of $Z$ that contain both $0$ and $2$.
  
  For $d \geq 2$, the $d$-dimensional vertically constant shift space $P^\dag(Z) \subset \Gamma^{\Z^d}$ is again sofic by Lemma~\ref{lem:BlackBox}.
  Each $y \in P^\dag(Z)$ is ultimately periodic to the left and right.
  However, given $m \in \N$ it is undecidable whether ${}^\infty 0 1^m 2^\infty \in Z$, since this is equivalent to the complement of the halting problem.
  Thus $Z$ does not satisfy the third countable cover condition, and by Corollary~\ref{cor:Main}, $P^\dag(Z)$ is not countably covered.
  It can also be shown that the language of $Z$ is computable.
  This example shows that the first two conditions, together with a computable language, do not imply the third.
\end{example}

With a similar construction, we also show that the last two conditions do not imply the first one.
Of course, since this example is not effectively closed, the corresponding vertically periodic shift space is not sofic.

\begin{example}
  \label{ex:Pi01Needed}
  Let $\Gamma = \{0, 1, 2, 3\}$, and consider the one-dimensional shift space $Z \subset \Gamma^\Z$ defined by the forbidden patterns $a b$ for $a > b \in \Gamma$, the pattern $0 1^m 2^{m+n} 3$ for each $m, n \in \N$ such that $T_m$ does not halt after $n$ steps, $0 1^m 2^{m+n+1}$ for each $m, n \in \N$ such that $T_m$ halts in exactly $n$ steps, as well as $0 1^m 2$ for each $m \in \N$ such that $T_m$ never halts.
  Then $Z$ consists of the configurations ${}^\infty 0 1^m 2^{m+n} 3^\infty$ where $T_m$ halts after $n$ steps, as well as those that do not contain both $0$ and $2$.
  The second countable cover condition is immediate, and the third is also easy to see.
  Since the pattern $0 1^m 2$ occurs in $Z$ if and only if $T_m$ eventually halts, the complement of the language is not computably enumerable, and hence $Z$ is not effectively closed.
\end{example}

A modified version of Example~\ref{ex:First} shows that vertical periodicity is essential for our proof technique.

\begin{example}
\label{ex:HasCover}
  Let $\Gamma = \{0,1\}$, and let $Z \subset \Gamma^\Z$ be as in Example~\ref{ex:First}.
  Consider the two-dimensional configuration $y \in \Gamma^{\Z^2}$ whose central horizontal row is $z = {}^\infty 0 . 1 1 0 1 0 0 1 0 0 0 1 \cdots$ and all other rows are filled with $0$s, and let $Y \subset \Gamma^{\Z^2}$ be the orbit closure of $y$.
  It consists of the orbit of $y$ and all configurations that contain at most one $1$.
  
  We claim that $Y$ is a countably covered sofic shift.
  For this, let $\Sigma$ be the tile set shown in Figure~\ref{fig:BoxAlph}(a), and let $X \subset \Sigma^{\Z^2}$ be the SFT consisting of those tilings where the contents of adjacent tiles match.
  The white regions with the symbols $\mathrm{N}$ and $\mathrm{S}$ are deemed to have distinct colors.
  This SFT is countable: if a configuration does not contain any parallel lines, then it consists of finitely many infinite regions separated by straight lines, and if it does contain parallel lines, then it contains at least one finite square, and thus an infinite sequence of growing squares, and is a translated version of the configuration in Figure~\ref{fig:BoxAlph}(b).
  
  Let $\pi : X \to \Gamma^{\Z^2}$ be the block code that sends each tile with a black dot to $1$, and the others to $0$.
  It is easy to see that $\pi(X) = Y$, and hence $Y$ is a countably covered sofic shift.
\end{example}

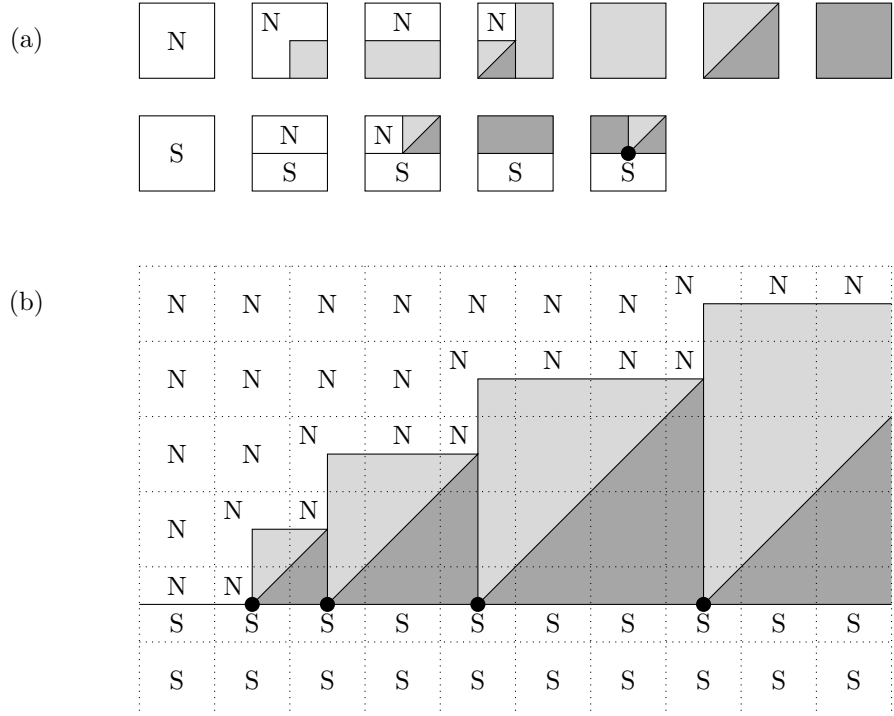
\begin{figure}[ht]
  \begin{center}
  \begin{tikzpicture}
  
  \node at (-1.5,0.5) {(a)};
  \node at (-1.5,-3) {(b)};
  
  \begin{scope}[shift={(0,0)}]
    \node at (0.5,0.5) {N};
    \draw (0,0) rectangle (1,1);
  \end{scope}
  
  \begin{scope}[shift={(1.5,0)}]
    \node at (0.25,0.75) {N};
    \fill [black!15] (0.5,0.5) rectangle (1,0);
    \draw (0.5,0) -- (0.5,0.5) -- (1,0.5);
    \draw (0,0) rectangle (1,1);
  \end{scope}
  
  \begin{scope}[shift={(3,0)}]
    \node at (0.5,0.75) {N};
    \fill [black!15] (0,0) rectangle (1,0.5);
    \draw (0,0.5) -- (1,0.5);
    \draw (0,0) rectangle (1,1);
  \end{scope}
  
  \begin{scope}[shift={(4.5,0)}]
    \node at (0.25,0.75) {N};
    \fill [black!15] (0.5,0) rectangle (1,1);
    \fill [black!15] (0,0) -- (0,0.5) -- (0.5,0.5);
    \fill [black!35] (0,0) -- (0.5,0) -- (0.5,0.5);
    \draw (0.5,0) -- (0.5,1);
    \draw (0,0.5) -- (0.5,0.5) -- (0,0);
    \draw (0,0) rectangle (1,1);
  \end{scope}
  
  \begin{scope}[shift={(6,0)}]
    \fill [black!15] (0,0) rectangle (1,1);
    \draw (0,0) rectangle (1,1);
  \end{scope}
  
  \begin{scope}[shift={(7.5,0)}]
    \fill [black!15] (0,0) -- (1,1) -- (0,1);
    \fill [black!35] (0,0) -- (1,1) -- (1,0);
    \draw (0,0) -- (1,1);
    \draw (0,0) rectangle (1,1);
  \end{scope}
  
  \begin{scope}[shift={(9,0)}]
    \fill [black!35] (0,0) rectangle (1,1);
    \draw (0,0) rectangle (1,1);
  \end{scope}
  
  \begin{scope}[shift={(0,-1.5)}]
    \node at (0.5,0.5) {S};
    \draw (0,0) rectangle (1,1);
  \end{scope}
  
  \begin{scope}[shift={(1.5,-1.5)}]
    \node at (0.5,0.25) {S};
    \node at (0.5,0.75) {N};
    \draw (0,0.5) -- (1,0.5);
    \draw (0,0) rectangle (1,1);
  \end{scope}
  
  \begin{scope}[shift={(3,-1.5)}]
    \node at (0.25,0.75) {N};
    \node at (0.5,0.25) {S};
    \fill [black!15] (0.5,0.5) -- (0.5,1) -- (1,1);
    \fill [black!35] (0.5,0.5) -- (1,0.5) -- (1,1);
    \draw (0,0.5) -- (0.5,0.5) -- (1,1);
    \draw (0.5,1) -- (0.5,0.5) -- (1,0.5);
    \draw (0,0) rectangle (1,1);
  \end{scope}
  
  \begin{scope}[shift={(4.5,-1.5)}]
    \node at (0.5,0.25) {S};
    \fill [black!35] (0,0.5) rectangle (1,1);
    \draw (0,0.5) -- (1,0.5);
    \draw (0,0) rectangle (1,1);
  \end{scope}
  
  \begin{scope}[shift={(6,-1.5)}]
    \node at (0.5,0.25) {S};
    \fill [black!35] (0,0.5) rectangle (0.5,1);
    \fill [black!35] (0.5,0.5) -- (1,1) -- (1,0.5);
    \fill [black!15] (0.5,0.5) -- (1,1) -- (0.5,1);
    \draw (0,0.5) -- (1,0.5);
    \draw (0.5,1) -- (0.5,0.5) -- (1,1);
    \fill (0.5,0.5) circle (0.1cm);
    \draw (0,0) rectangle (1,1);
  \end{scope}

  \begin{scope}[yshift=-8.5cm]
  
    \fill [black!15] (1.5,1.5) -- ++(0,1) -- ++(1,0);
    \fill [black!15] (2.5,1.5) -- ++(0,2) -- ++(2,0);
    \fill [black!15] (4.5,1.5) -- ++(0,3) -- ++(3,0);
    \fill [black!15] (7.5,1.5) -- ++(0,4) -- ++(2.5,0) -- ++(0,-1.5);
    
    \fill [black!35] (1.5,1.5) -- ++(1,0) -- ++(0,1);
    \fill [black!35] (2.5,1.5) -- ++(2,0) -- ++(0,2);
    \fill [black!35] (4.5,1.5) -- ++(3,0) -- ++(0,3);
    \fill [black!35] (7.5,1.5) -- ++(2.5,0) -- ++(0,2.5);
    
    \draw (10,1.5) -- (1.5,1.5) -- ++(0,1) -- ++(1,0) -- ++(0,1) -- ++(2,0) -- ++(0,1) -- ++(3,0) -- ++(0,1) -- ++(2.5,0);
    \draw (1.5,1.5) -- ++(1,1) -- ++(0,-1) -- ++(2,2) -- ++(0,-2) -- ++(3,3) -- ++(0,-3) -- ++(2.5,2.5);
    \draw (0,1.5) -- (1.5,1.5);
    
    \foreach \x in {1,2,4,7}{
      \fill (\x+0.5,1.5) circle (0.1cm);
    }
    
    \foreach \x in {0,...,9}{
    	\node at (\x+0.5,0.5) {S};
    	\node at (\x+0.5,1.25) {S};
    }
    \foreach \x/\y in {0/2,0/3,0/4,0/5,1/3,1/4,1/5,2/4,2/5,3/4,3/5,4/5,5/5,6/5}{
      \node at (\x+0.5,\y+0.5) {N};
    }
    \foreach \x/\y in {0/1,3/3,5/4,6/4,8/5,9/5}{
      \node at (\x+0.5,\y+0.75) {N};
    }
    \foreach \x/\y in {1/1,1/2,2/2,2/3,4/3,4/4,7/4,7/5}{
      \node at (\x+0.25,\y+0.75) {N};
    }
    
    \draw [dotted] (0,0) grid (10,6);
    
  \end{scope}
  
  \end{tikzpicture}
  \end{center}
  \caption{(a) The tile set of Example~\ref{ex:HasCover}, and (b) an example configuration. Dotted lines denote tile borders.}
  \label{fig:BoxAlph}
\end{figure}

We show that the converse of Corollary~\ref{cor:Main} holds in a particular subclass of shift spaces, namely, those whose projective subdynamics is contained in some one-dimensional countable sofic shift.
In \cite{SaTo13pub}, such systems were said to have the \emph{bounded signal property}.
Before this, we give a few examples of two-dimensional countably covered sofic shifts that do not belong to this class, both to show that the result is not optimal, and to highlight some techniques of the construction.

\begin{example}
  \label{ex:Grid}
  Consider the set of tiles shown in Figure~\ref{fig:GridAlph}(a), and the SFT $X$ where the lines and colors of neighboring tiles must match.
  The configurations of $X$ are grids formed by infinitely many squares of equal size, or `degenerate' configurations containing infinite regions of the same color; see Figure~\ref{fig:GridAlph}(b) for an example.
  It is easy to see that $X$ is countable.
  
  Consider the map $\phi : X \to \{0,1\}^{\Z^2}$ that sends to $1$ those tiles that contain a vertical line (the two rightmost tiles in Figure~\ref{fig:GridAlph}), and to $0$ all other tiles.
  Then $\phi(X)$ is a countably covered vertically constant sofic shift.
  The projective subdynamics $P(\phi(X))$ consists of the periodic configurations ${}^\infty (0^n 1)^\infty$ for all $n \in \N$, the ultimately periodic configuration ${}^\infty 0 1 0^\infty$, and the uniform configuration ${}^\infty 0^\infty$, as well as their shifts.
\end{example}

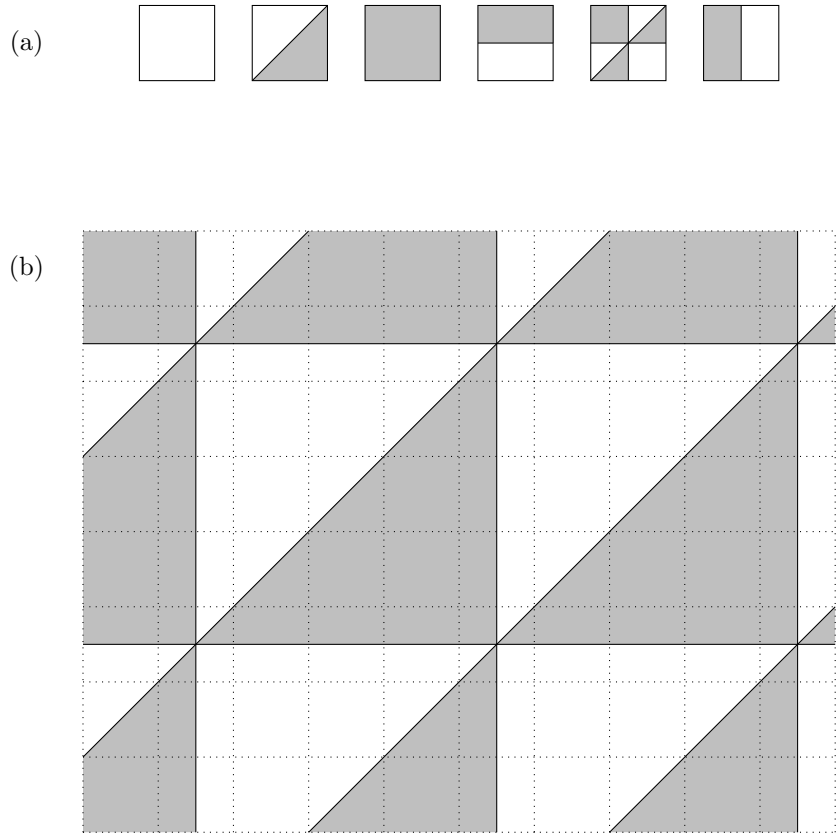
\begin{figure}[ht]
  \begin{center}
  \begin{tikzpicture}
  
  \node at (-1.5,0.5) {(a)};
  \node at (-1.5,-2.5) {(b)};
  
  \begin{scope}[shift={(0,0)}]
    \draw (0,0) rectangle (1,1);
  \end{scope}
  
  \begin{scope}[shift={(1.5,0)}]
    \fill [black!25] (0,0) -- (1,0) -- (1,1);
    \draw (0,0) -- (1,1);
    \draw (0,0) rectangle (1,1);
  \end{scope}
  
  \begin{scope}[shift={(3,0)}]
    \fill [black!25] (0,0) rectangle (1,1);
    \draw (0,0) rectangle (1,1);
  \end{scope}
  
  \begin{scope}[shift={(4.5,0)}]
    \fill [black!25] (0,0.5) rectangle (1,1);
    \draw (0,0.5) -- (1,0.5);
    \draw (0,0) rectangle (1,1);
  \end{scope}
  
  \begin{scope}[shift={(6,0)}]
    \fill [black!25] (0,0) -- (0.5,0) -- (0.5,0.5) -- (1,0.5) -- (1,1);
    \fill [black!25] (0,0.5) rectangle (0.5,1);
    \draw (0,0) -- (1,1);
    \draw (0,0.5) -- (1,0.5);
    \draw (0.5,0) -- (0.5,1);
    \draw (0,0) rectangle (1,1);
  \end{scope}
  
  \begin{scope}[shift={(7.5,0)}]
    \fill [black!25] (0,0) rectangle (0.5,1);
    \draw (0.5,0) -- (0.5,1);
    \draw (0,0) rectangle (1,1);
  \end{scope}

  \begin{scope}[xshift=-0.75cm,yshift=-10cm]
  
    \fill [black!25] (0,6.5) rectangle (1.5,8);
    \fill [black!25] (0,2.5) -- (1.5,2.5) -- (1.5,6.5) -- (0,5);
    \fill [black!25] (0,0) -- (1.5,0) -- (1.5,2.5) -- (0,1);
    
    \fill [black!25] (1.5,6.5) -- (5.5,6.5) -- (5.5,8) -- (3,8);
    \fill [black!25] (1.5,2.5) -- (5.5,2.5) -- (5.5,6.5);
    \fill [black!25] (3,0) -- (5.5,0) -- (5.5,2.5);
    
    \fill [black!25] (5.5,6.5) -- (9.5,6.5) -- (9.5,8) -- (7,8);
    \fill [black!25] (5.5,2.5) -- (9.5,2.5) -- (9.5,6.5);
    \fill [black!25] (7,0) -- (9.5,0) -- (9.5,2.5);
    
    \fill [black!25] (9.5,6.5) -- (10,6.5) -- (10,7);
    \fill [black!25] (9.5,2.5) -- (10,2.5) -- (10,3);
  
    \draw (1.5,0) -- (1.5,8);
    \draw (5.5,0) -- (5.5,8);
    \draw (9.5,0) -- (9.5,8);
    
    \draw (0,2.5) -- (10,2.5);
    \draw (0,6.5) -- (10,6.5);
    
    \draw (0,5) -- (3,8);
    \draw (0,1) -- (7,8);
    \draw (3,0) -- (10,7);
    \draw (7,0) -- (10,3);
    
    \draw [dotted] (0,0) grid (10,8);
    
  \end{scope}
  
  \end{tikzpicture}
  \end{center}
  \caption{(a) The alphabet of the grid shift of Example~\ref{ex:Grid}, and (b) an example configuration. Dotted lines denote tile borders.}
  \label{fig:GridAlph}
\end{figure}

\begin{example}
  \label{ex:Counting}
  Consider the one-dimensional shift space $X \subset \{0,1\}^\Z$ defined by forbidding those words over $\{0,1\}$ that, for some $n \in \N$, contain $1 0^n 1$ as a subword and contain more than $n$ occurrences of $1$.
  Then $X$ contains precisely those configurations where the number of $1$s is strictly less than the distance between any two occurrences of $1$.
  Since every configuration of $X$ contains a finite number of $1$s, it is a countable space.

\begin{figure}[ht]
  \begin{center}
  \begin{tikzpicture}

  \node at (-1,-0.25) {(a)};
  \node at (-1,-3.75) {(b)};
  \node at (-1,-6.5) {(c)};
  
  \begin{scope}[shift={(0,0)}]
    \draw (0,0) rectangle (1,1);
  \end{scope}
  
  \begin{scope}[shift={(1.5,0)}]
    \fill [black!15] (0,0) rectangle (1,1);
    \draw (0,0) rectangle (1,1);
  \end{scope}
  
  \begin{scope}[shift={(3,0)}]
    \fill [black!35] (0,0) -- (1,0) -- (1,1);
    \draw [dashed] (0,0) -- (1,1);
    \draw (0,0) rectangle (1,1);
  \end{scope}
  
  \begin{scope}[shift={(4.5,0)}]
    \fill [black!15] (0,0) -- (0,1) -- (1,1);
    \fill [black!35] (0,0) -- (1,0) -- (1,1);
    \draw [dashed] (0,0) -- (1,1);
    \draw (0,0) rectangle (1,1);
  \end{scope}
  
  \begin{scope}[shift={(6,0)}]
    \fill [black!35] (0,0) rectangle (1,1);
    \draw (0,0) rectangle (1,1);
  \end{scope}
  
  \begin{scope}[shift={(7.5,0)}]
    \fill [black!35] (0,0.5) rectangle (1,1);
    \draw [very thick] (0,0.5) -- (1,0.5);
    \draw (0,0) rectangle (1,1);
  \end{scope}

  \begin{scope}[shift={(9,0)}]
    \fill [black!35] (0,0.5) rectangle (1,1);
    \fill [black!15] (0,0) rectangle (1,0.5);
    \draw [very thick] (0,0.5) -- (1,0.5);
    \draw (0,0) rectangle (1,1);
  \end{scope}
  
  \begin{scope}[shift={(0,-1.5)}]
    \fill [black!35] (0,0) rectangle (0.5,1);
    \draw [very thick] (0.5,0) -- (0.5,1);
    \draw (0,0) rectangle (1,1);
  \end{scope}

  \begin{scope}[shift={(1.5,-1.5)}]
    \fill [black!35] (0,0) rectangle (0.5,1);
    \fill [black!15] (0.5,0) rectangle (1,1);
    \draw [very thick] (0.5,0) -- (0.5,1);
    \draw (0,0) rectangle (1,1);
  \end{scope}
  
  \begin{scope}[shift={(3,-1.5)}]
    \fill [black!35] (0,0) -- (0.5,0.5) -- (0.5,0);
    \fill [black!15] (0.5,0) rectangle (1,1);
    \draw [dashed] (0,0) -- (0.5,0.5);
    \draw [very thick] (0.5,0) -- (0.5,1);
    \draw (0,0) rectangle (1,1);
  \end{scope}
  
  \begin{scope}[shift={(4.5,-1.5)}]
    \fill [black!15] (0.5,0) rectangle (1,1);
    \draw [very thick] (0.5,0) -- (0.5,1);
    \draw (0,0) rectangle (1,1);
  \end{scope}
  
  \begin{scope}[shift={(6,-1.5)}]
    \fill [black!35] (0,0) -- (0.5,0) -- (0.5,0.5) -- (1,0.5) -- (1,1);
    \fill [black!35] (0,0.5) rectangle (0.5,1);
    \fill [black!15] (0,0) -- (0.5,0.5) -- (0,0.5);
    \draw [very thick] (0.5,0) -- (0.5,1);
    \draw [very thick] (0,0.5) -- (1,0.5);
    \draw [dashed] (0,0) -- (1,1);
    \draw (0,0) rectangle (1,1);
  \end{scope}
  
  \begin{scope}[shift={(7.5,-1.5)}]
    \fill [black!35] (0,0) -- (0.5,0) -- (0.5,0.5) -- (1,0.5) -- (1,1);
    \fill [black!35] (0,0.5) rectangle (0.5,1);
    \fill [black!15] (0,0) -- (0,0.5) -- (0.5,0.5) -- (0.5,1) -- (1,1);
    \fill [black!15] (0.5,0) rectangle (1,0.5);
    \draw [very thick] (0.5,0) -- (0.5,1);
    \draw [very thick] (0,0.5) -- (1,0.5);
    \draw [dashed] (0,0) -- (1,1);
    \draw (0,0) rectangle (1,1);
  \end{scope}
  
  \begin{scope}[shift={(9,-1.5)}]
    \fill [black!35] (0.5,0.5) -- (1,0.5) -- (1,1);
    \fill [black!35] (0,0.5) rectangle (0.5,1);
    \fill [black!15] (0.5,0) rectangle (1,0.5);
    \fill [black!15] (0.5,0.5) -- (0.5,1) -- (1,1);
    \draw [very thick] (0.5,0) -- (0.5,1);
    \draw [very thick] (0,0.5) -- (1,0.5);
    \draw [dashed] (0.5,0.5) -- (1,1);
    \draw (0,0) rectangle (1,1);
  \end{scope}

  \begin{scope}[shift={(0,-3.5)}]
    \draw (0,0) rectangle (1,1);
  \end{scope}
  
  \begin{scope}[shift={(1.5,-3.5)}]
    \fill [black!25] (0,0) rectangle (1,1);
    \draw (0,0) rectangle (1,1);
  \end{scope}
  
  \begin{scope}[shift={(3,-3.5)}]
    \fill [black!25] (0,0.5) rectangle (1,1);
    \draw [very thick] (0,0.5) -- (1,0.5);
    \draw (0,0) rectangle (1,1);
  \end{scope}
  
  \begin{scope}[shift={(4.5,-3.5)}]
    \fill [black!25] (0,0) rectangle (1,0.5);
    \draw [dashed] (0,0.5) -- (1,0.5);
    \draw (0,0) rectangle (1,1);
  \end{scope}

  \begin{scope}[shift={(6,-3.5)}]
    \draw [very thick] (0.5,0) -- (0.5,1);
    \draw (0,0) rectangle (1,1);
  \end{scope}
  
  \begin{scope}[shift={(7.5,-3.5)}]
    \fill [black!25] (0,0) rectangle (1,0.5);
    \draw [very thick] (0.5,0) -- (0.5,1);
    \draw [dashed] (0,0.5) -- (1,0.5);
    \draw (0,0) rectangle (1,1);
  \end{scope}
  
  \begin{scope}[shift={(9,-3.5)}]
    \fill [black!25] (0,0) rectangle (1,1);
    \draw [very thick] (0.5,0) -- (0.5,1);
    \draw (0,0) rectangle (1,1);
  \end{scope}
  
  \begin{scope}[shift={(0,-5)}]
    \draw [ultra thick,double] (0.5,0) -- (0.5,1);
    \draw (0,0) rectangle (1,1);
  \end{scope}
  
  \begin{scope}[shift={(1.5,-5)}]
    \fill [black!25] (0.5,0) rectangle (1,0.5);
    \draw [dashed] (0.5,0.5) -- (1,0.5);
    \draw [ultra thick,double] (0.5,0) -- (0.5,1);
    \draw (0,0) rectangle (1,1);
  \end{scope}
  
  \begin{scope}[shift={(3,-5)}]
    \fill [black!25] (0,0) rectangle (0.5,0.5);
    \fill [black!25] (0.5,0) rectangle (1,1);
    \draw [dashed] (0,0.5) -- (0.5,0.5);
    \draw [ultra thick,double] (0.5,0) -- (0.5,1);
    \draw (0,0) rectangle (1,1);
  \end{scope}
  
  \begin{scope}[shift={(4.5,-5)}]
    \fill [black!25] (0,0) rectangle (1,1);
    \draw [ultra thick,double] (0.5,0) -- (0.5,1);
    \draw (0,0) rectangle (1,1);
  \end{scope}
  
  \begin{scope}[shift={(6,-5)}]
    \fill [black!25] (0,0.5) rectangle (1,1);
    \draw [very thick] (0,0.5) -- (1,0.5);
    \draw [ultra thick,double] (0.5,0) -- (0.5,1);
    \draw (0,0) rectangle (1,1);
  \end{scope}
  
  \begin{scope}[shift={(7.5,-5)}]
    \fill [black!25] (0,0.5) rectangle (1,1);
    \draw [very thick] (0,0.5) -- (1,0.5);
    \draw [very thick] (0.5,0) -- (0.5,1);
    \draw (0,0) rectangle (1,1);
  \end{scope}

  \begin{scope}[shift={(0,-7)}]
    \draw (0,0) rectangle (1,1);
  \end{scope}
  
  \begin{scope}[shift={(1.5,-7)}]
    \fill [black!25] (0.5,0) rectangle (1,1);
    \draw [dashed] (0.5,0) -- (0.5,1);
    \draw (0,0) rectangle (1,1);
  \end{scope}
  
  \begin{scope}[shift={(3,-7)}]
    \fill [black!25] (0,0) rectangle (1,1);
    \draw (0,0) rectangle (1,1);
  \end{scope}
  
  \begin{scope}[shift={(4.5,-7)}]
    \fill [black!25] (0,0) rectangle (0.5,1);
    \draw [very thick] (0.5,0) -- (0.5,1);
    \draw (0,0) rectangle (1,1);
  \end{scope}
  
  \begin{scope}[shift={(6,-7)}]
    \draw [very thick] (0.5,0) -- (0.5,1);
    \draw (0,0) rectangle (1,1);
  \end{scope}
  
  \end{tikzpicture}
  \end{center}
  \caption{The three layers of the alphabet of the shift space $Z$ of Example~\ref{ex:Counting}: (a) the grid layer, (b) the counter layer, and (c) the synchronization layer.}
  \label{fig:CountingAlph}
\end{figure}
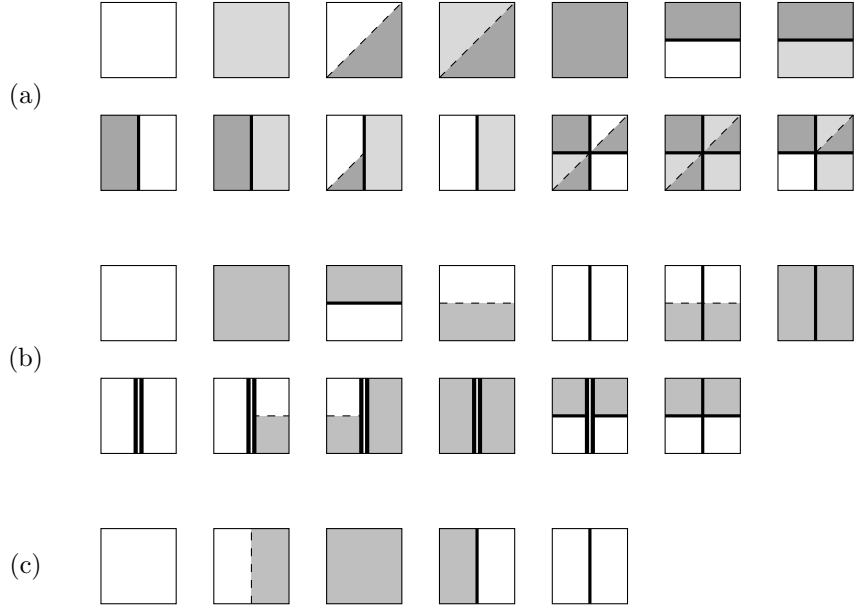

\begin{figure}[ht]
  \begin{center}
  \begin{tikzpicture}
  
    \fill [black!15] (0,0) rectangle (5.5,8);
    \fill [black!15] (7.5,0) rectangle (10,8);
  
    \fill [black!35] (0,6.5) rectangle (1.5,8);
    \fill [black!35] (0,2.5) -- (1.5,2.5) -- (1.5,6.5) -- (0,5);
    \fill [black!35] (0,0) -- (1.5,0) -- (1.5,2.5) -- (0,1);
    
    \fill [black!35] (1.5,6.5) -- (5.5,6.5) -- (5.5,8) -- (3,8);
    \fill [black!35] (1.5,2.5) -- (5.5,2.5) -- (5.5,6.5);
    \fill [black!35] (3,0) -- (5.5,0) -- (5.5,2.5);
    
    \fill [black!35] (5.5,6.5) -- (7.5,6.5) -- (7.5,8) -- (7,8);
    \fill [black!35] (5.5,2.5) -- (7.5,2.5) -- (7.5,4.5);
    \fill [black!35] (7,0) -- (7.5,0) -- (7.5,0.5);
    
    \fill [black!35] (7.5,6.5) -- (10,6.5) -- (10,8) -- (9,8);
    \fill [black!35] (7.5,2.5) -- (10,2.5) -- (10,5);
    \fill [black!35] (9,0) -- (10,0) -- (10,1);
  
    \draw [very thick] (1.5,0) -- (1.5,8);
    \draw [very thick] (5.5,0) -- (5.5,8);
    \draw [very thick] (7.5,0) -- (7.5,8);
    
    \draw [very thick] (0,2.5) -- (10,2.5);
    \draw [very thick] (0,6.5) -- (10,6.5);
    
    \draw [dashed] (0,5) -- (3,8);
    \draw [dashed] (0,1) -- (7,8);
    \draw [dashed] (3,0) -- (7.5,4.5);
    \draw [dashed] (7,0) -- (7.5,0.5);
    \draw [dashed] (7.5,6.5) -- (9,8);
    \draw [dashed] (7.5,2.5) -- (10,5);
    \draw [dashed] (9,0) -- (10,1);
    
    \draw [dotted] (0,0) grid (10,8);
    
  \end{tikzpicture}
  \end{center}
  \caption{An example configuration of the grid layer is Example~\ref{ex:Counting}.}
  \label{fig:GridExample}
\end{figure}
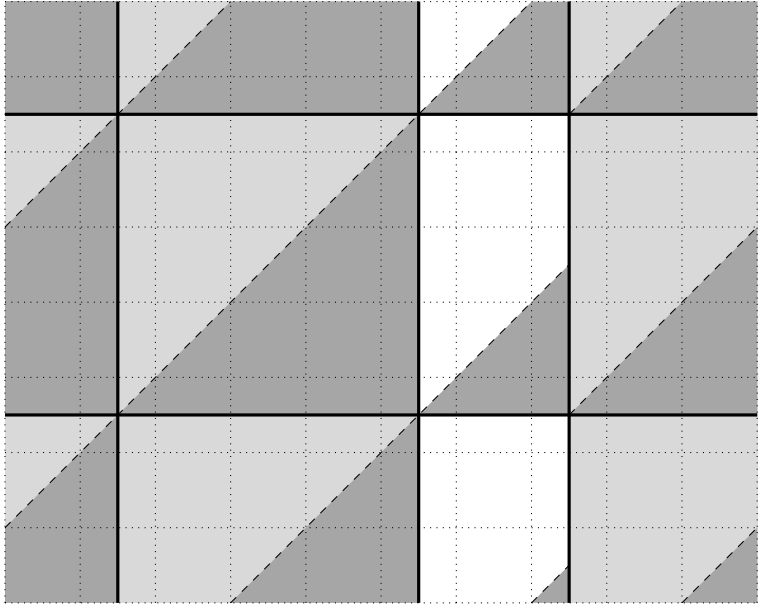

\begin{figure}[ht]
  \begin{center}
  \begin{tikzpicture}
  
    \fill [black!25] (0,6.5) -- (10,6.5) -- (10,8) -- (1.5,8) -- (1.5,7.5) -- (0,7.5);
    \fill [black!25] (0,2.5) -- (10,2.5) -- (10,5.5) -- (7.5,5.5) -- (7.5,4.5) -- (1.5,4.5) -- (1.5,3.5) -- (0,3.5);
    \fill [black!25] (1.5,0) -- (10,0) -- (10,1.5) -- (7.5,1.5) -- (7.5,0.5) -- (1.5,0.5);
    
    \draw [very thick] (0,2.5) -- (10,2.5);
    \draw [very thick] (0,6.5) -- (10,6.5);
    
    \draw [dashed] (0,7.5) -- (1.5,7.5);
    \draw [dashed] (0,3.5) -- (1.5,3.5);
    
    \draw [dashed] (1.5,4.5) -- (7.5,4.5);
    \draw [dashed] (1.5,0.5) -- (7.5,0.5);
    
    \draw [dashed] (7.5,5.5) -- (10,5.5);
    \draw [dashed] (7.5,1.5) -- (10,1.5);
    
    \draw [dotted] (0,0) grid (10,8);
  
    \draw [ultra thick,double] (1.5,0) -- (1.5,8);
    \draw [very thick] (5.5,0) -- (5.5,8);
    \draw [ultra thick,double] (7.5,0) -- (7.5,8);
    
  \end{tikzpicture}
  \end{center}
  \caption{An example configuration of the counter layer of Example~\ref{ex:Counting}. The thick lines coincide with those of Figure~\ref{fig:GridExample}.}
  \label{fig:CounterExample}
\end{figure}
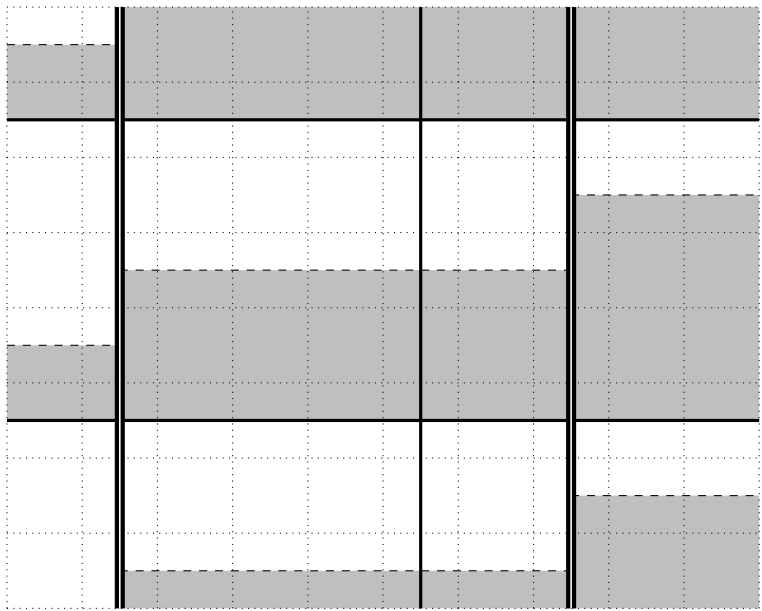

  Let $Y  = P^\dag(X) \subset \{0,1\}^{\Z^2}$.
  We show that $Y$ is a countably covered sofic shift by constructing a countable SFT $Z$ and a factor map $\phi : Z \to Y$.
  The alphabet of $Z$ consists of three layers: the \emph{grid layer} (which is somewhat similar to the SFT of Example~\ref{ex:Grid} but more complex), the \emph{counter layer}, and the \emph{synchronization layer}.
  The idea is that the height of the grid cells on the grid layer is determined by the minimum distance between two $1$s, the counter layer uses this information to globally bound the number of $1$s, and the synchronization layer enforces some additional constraints to ensure countability.
  The tiles of the layers are shown in Figure~\ref{fig:CountingAlph}, and example configurations of the first two layers are shown in Figure~\ref{fig:GridExample} and Figure~\ref{fig:CounterExample}.
  The lines and colors of each layer must be correctly matched in neighboring tiles.
  Furthermore, we have the following rules between the layers:
  \begin{enumerate}
  \item The thick vertical lines of the grid layer must match the thick vertical lines of the counter layer and the synchronization layer.
  \item The thick horizontal lines of the grid layer must match the thick horizontal lines of the counter layer.
  \item If the dashed vertical line of the synchronization layer crosses the dashed diagonal line of the grid layer, then the dashed horizontal line of the counter layer must be present at that position as well.
  \item A double vertical line on the counter layer must be paired with a tile of the grid layer that has a light gray region in its right half.
  \item A vertical line of the grid layer with a white region on its left must be paired with a double vertical line on the counter layer.
  \end{enumerate}
  Otherwise, the three layers are independent.
  
  These rules force the following structure to $Z$.
  On the grid layer, we have an infinite grid of rectangular areas which are at least as high as they are wide (possibly with degenerate infinite grid cells).
  Their top halves are colored either white (if they are higher than they are wide) or light gray (if they are square-shaped).
  The colors of a single vertical column of grid cells must match.
  In particular, in a configuration with only finite grid cells and at least one light gray column, the heights of all grid cells are equal.
  On the counter layer, each horizontal row of grid cells of height $h$ contains one counter, which may be incremented at most $h-1$ times at the vertical borders between grid cells.
  A vertical border is a doubled line if an increment happens at that position.
  By rule 3, the counters of different rows are synchronized by the vertical signals of the synchronization layer, and by rule 4, an incrementation of the counter causes the next column of grid cells to be light gray, and hence square-shaped.
  Conversely, rule 5 ensures that a non-square column of grid cells is accompanied by an increase of the counter, so only finitely many columns of grid cells can be non-square.
  In particular, since the counter can be incremented finitely many times, a configuration can contain only finitely many white columns.
  
  We claim that the SFT $Z$ is countable.
  Consider first a configuration that contains a finite grid cell, wich forces the entire configuration to consist of a grid of finite cells.
  There are countably many choices of the height of the grid cells, and after that, countably many choices for the finitely many columns of non-square grid cells and the values of the counters.
  Other parts of the configuration are forced by these choices.
  An infinite grid cell cannot have finite width or height, so there are countably many choices for configurations containing them as well.
  
  The block code $\phi$ sends each double vertical line on the counter layer to $1$, and all other tiles to $0$.
  Thus, in the image of a configuration with height-$h$ rectangles, there can be at most $h-1$ columns of $1$s (since counters can be incremented at most $h-1$ times), and the distance between consecutive columns of $1$s is at least $h$ (since each incrementation must be followed by a column of square grid cells).
  With the help of non-square grid cells, the distance can in fact be any number at least $h$, independently for each pair of consecutive $1$s.
  Thus $\phi(Z) = Y$.
\end{example}

\section{Partial Converse}

We now prove a partial converse to Corollary~\ref{cor:Main} in a restricted class of shift spaces.
The proof is a relatively complicated construction spanning most of the remaining article.
We begin with some closure properties of the class of shift spaces that satisfy the countable cover conditions.

\begin{lemma}
  \label{lem:Simplification}
  Let $W \subset \Sigma^\Z$ be an edge shift and $Z \subset \Gamma^\Z$ a sofic shift, and let $\psi : W \to Z$ be a right-resolving factor map.
  Let $Y \subset Z$ be a shift space that satisfies the countable cover conditions.
  Then $X = \psi^{-1}(Y)$ satisfies the countable cover conditions as well.
\end{lemma}

\begin{proof}
  As a preimage of an effectively closed shift space, $X$ is effectively closed.

  Let $x \in X$ be arbitrary, and denote its ultimately periodic image by $\phi(x) = {}^\infty u v w^\infty$, where $u, v, w \in \Gamma^+$.
  Let $G = (V, E)$ be the graph associated to the edge shift $W$ (with $E = \Sigma$).
  Denote by $\rho_u : \Sigma \to \Sigma \cup \{\#\}$ the partial transition function induced by the word $u$, which is well-defined by the right-resolvingness of $\psi$, and similarly for $w$.
  If $x_i = s \in \Sigma$ and $\phi(x)_{[i, i+|G||u|-1]} = u^{|G|}$, then there exist $m, n \leq |G|$ such that $x_{i+m|u|} = \rho^m_u(s) = \rho^{m+k n}_u(s) = x_{i+(m + k n)|u|}$ for all $k \in \N$.
  This implies that $x$ is ultimately periodic to the left with period $n|u|$.
  Similarly, it is ultimately periodic to the right with period $n' |w|$ for some $n' \leq |G|$.
  This implies that $X$ satisfies the second countable cover condition.

  Let $u, v, w \in \Gamma^+$ be arbitrary.
  The condition ${}^\infty u v w^\infty \in X$ is equivalent to $\psi({}^\infty u v w^\infty) \in Y$.
  This configuration is ultimately periodic, and the periods and transient part are easily computable from $u$, $v$ and $w$.
  Hence $X$ satisfies the third countable cover condition.
\end{proof}

\begin{lemma}
  \label{lem:SubSFT}
  Let $Y \subset \Gamma^\Z$ be a shift space that satisfies the countable cover conditions, and let $Z \subset \Gamma^\Z$ be an SFT.
  Then $X = Y \cap Z$ satisfies the countable cover conditions as well.
\end{lemma}

\begin{proof}
  The first countable cover condition holds, since effectively closed shifts are closed under finite intersection.
  The second condition is immediate, since $X \subset Y$.
  Given $u, v, w \in \Gamma^+$, we can decide ${}^\infty u v w^\infty \in Y$ by assumption, and ${}^\infty u v w^\infty \in Z$ is decidable since $Z$ is defined by finitely many local rules.
  Hence $X$ satisfies the third countable cover condition.
\end{proof}

\begin{theorem}
  \label{thm:Converse}
  Let $Z \subset \Gamma^\Z$ be a countable sofic shift, let $d \geq 2$, and let $Y \subset \Gamma^{\Z^d}$ be a vertically constant countable shift space with $P(Y) \subset Z$.
  If $Y$ satisfies the countable cover conditions, then it is countably covered.
\end{theorem}

\begin{proof}
  We prove the result in the case $d = 2$.
  The general case follows easily by considering an SFT cover which is constant in each dimension except the first two.

  We first simplify the structure of $Z$ in order to make it easier to handle in the construction.
  Let $W \subset \Sigma^\Z$ be an edge shift and $\psi : W \to Z$ a right-resolving factor map as given by Lemma~\ref{lem:CountableStructure}.
  By Lemma~\ref{lem:Simplification}, the $\psi$-preimage of $P(Y)$ satisfies the countable cover conditions.
  Let $Y' \subset \Gamma^{\Z^2}$ be a vertically constant shift space with $P(Y') = \psi^{-1}(P(Y))$.
  If $Y'$ is countably covered, then so is $Y = \psi(Y')$.
  Hence we can safely assume that $Z$ is an edge shift satisfying the conditions of Lemma~\ref{lem:CountableStructure} by replacing it with $Z'$ if necessary.
  Thus $Z$ is a finite (not necessarily disjoint) union of SFTs of the form
  \begin{equation}
  \label{eq:Component}
    \mathrm{Cl}( \{ {}^\infty u_0 v_1 u_1^{n_1} v_2 u_2^{n_2} \cdots v_k u_k^\infty \;|\; n_1, \ldots, n_k \in \N \} )
  \end{equation}
  for some \emph{rank} $k \in \N$ and words $u_i, v_i \in \Gamma^+$, where the $u_i$ come from disjoint cycles in the graph and hence do not share any symbols.
  The $v_i$ consist of the paths between these cycles, so they do not share symbols with any other $v_j$ or $u_j$.
  The $u_i$ are called the \emph{period words} and the $v_i$ are called the \emph{period breakers}.
  By Lemma~\ref{lem:SubSFT}, we can consider each of these components separately, and take the disjoint union of the resulting two-dimensional countable SFTs to obtain a countable cover for $Y$.
  From now on, we assume that $Z$ has the form~\eqref{eq:Component} for some rank $k \in \N$ and words $u_i, v_i \in \Gamma^+$ that share no symbols.
  We may also assume that for all $N \in \N$, the shift space $P(Y)$ contains a configuration ${}^\infty u_0 v_1 u_1^{n_1} v_2 u_2^{n_2} \cdots v_k u_k^\infty$ with $n_i \geq N$ for all $1 \leq i \leq k$; if this is not the case, we can replace $Z$ by a finite union of components with rank $k-1$ that use some words of the form $v_i u_i^{n_i} v_{i+1}$ as period breakers.
  In particular, we have ${}^\infty u_i v_{i+1} u_{i+1}^\infty \in P(Y)$ for each $i$.
  
  We now begin the construction of the countable SFT $X$.
  It contains a number of layers, one for each pair of indices $1 \leq i < j \leq k$, corresponding to the period breakers $v_i$ and $v_j$.
  This layer is denoted by $X_{i,j}$, and its purpose is to run a computation verifying the correctness of a configuration that contains both $v_i$ and $v_j$.
  The SFT $X$ also contains a vertically constant layer whose rows come from $Z$, called the \emph{image layer}, which is further constrained by the other layers.
  The block code $\phi : X \to \Gamma^{\Z^2}$ is simply a projection onto the image layer.
  
  Each layer $X_{i,j}$ has the same basic structure, which we break into three sub-layers.
  The first layer is called the \emph{verification machine layer}.
  Let $M_{i,j} = (k, Q, \Sigma, \delta, q_0, q_f)$ be a deterministic integer $k$-counter machine with one oracle counter.
  Let $V_{i,j}$ be the SFT defined in Section~\ref{sec:Machines} that simulates the machine $M_{i,j}$.
  It contains a computation cone that expands to the left and right, a special read-only counter that always maintains the value $0$, and an oracle layer that can be probed by the oracle counter.
  We modify $V_{i,j}$ as follows.
  First, the oracle layer is replaced by the image layer of $X$, so that the simulated machine is able to read its contents.
  When the zig-zag head enters the final state $q_f$, it no longer makes a new sweep, but stays at the origin indefinitely.
  We also introduce a new symbol $\#$, which is required to fill an entire horizontal row of the layer wherever it is present.
  If the next row contains the zig-zag head, it must be in the initial state $q_0$, with each counter set to $0$ and the computation cone having width $1$.
  If the previous row contains the zig-zag head, it must be in the final state.
  There are no other restrictions on $\#$.
  This means that the new symbol can effectively reset the computation after it has halted.
  On each horizontal row of the verification machine layer (except those that are filled with $\#$), we require that the anchor counter is placed on the leftmost symbol of $v_i$ on the image track.
  Since this symbol occurs only once in $v_i$, and does not occur in any word $v_p$ for $p \neq i$ or $u_p$ for $1 \leq p \leq k$, this uniquely determines the position of the anchor counter as a function of the image layer.
  
  The counter machine $M_{i,j}$ behaves as follows.
  First, it reads the entire string $w = v_i u_i^{n_i} v_{i+1} \cdots v_j \in \Sigma^+$ between the words $v_i$ and $v_j$ on the image layer (if $v_j$ is not present, this computation never halts).
  It then checks whether ${}^\infty u_{i-1} w u_j^\infty \in P(Y)$, which is computable by assumption.
  If this holds, then $M_{i,j}$ enters the final state and suspends the computation.
  If ${}^\infty u_{i-1} w u_j^\infty \notin P(Y)$, then $M_{i,j}$ computes a number $p \in \N$ such that $u_{i-1}^p w u_j^p$ does not occur in $P(Y)$, which is possible since $P(Y)$ has a computable language.
  Then, $M_{i,j}$ reads the $|u_{i-1}|^p$ symbols to the left of $v_i$, and the $|u_j|^p$ symbols to the right of $v_j$, obtaining an extended string $u w v \in \Sigma^+$.
  If $u w v$ occurs in $P(Y)$, then $M_{i,j}$ enters the final state and suspends the computation.
  If $u w v$ does not occur in $P(Y)$, then $M_{i,j}$ halts and produces a tiling error.
  
  A moment's reflection should convince the reader that if each machine $M_{i,j}$ for $1 \leq i < j \leq k$ is able to complete its computation in a configuration of $X$, then the image layer is a configuration of $P^\dag(Y)$.
  However, at this point of the construction it is not guaranteed that every valid configuration contains such computations, since the machines may stay in the final state $q_f$ on each row.
  Another issue is that the machines may perform their computations infinitely many times, with a row of $\#$-symbols between each run, and since the zig-zag head may wait in the final state $q_f$ an indeterminate number of steps before each reset, the number of such configurations is uncountable.
  The remaining layers of $X_{i,j}$ will fix these issues by forcing the $\#$-rows to occur periodically.
  
  The second sub-layer of $X_{i,j}$ is the \emph{grid}, which is exactly equivalent to the SFT of Example~\ref{ex:Grid}.
  We require that the leftmost symbol of $v_i$ and the leftmost symbol of $v_j$ (which are unique to these words and occur in them only once) are paired with vertical lines of the grid, and none of the symbols between them (the remaining symbols of $v_i$ and $v_j$, as well as all symbols of $u_p$ for $i \leq p < j$ and $v_p$ for $i < p < j$) are paired with vertical lines of the grid.
  This completely determines the locations of the vertical lines of the grid as a function of the image layer.
  Those grid cells that are bordered by $v_i$ and $v_j$ are called \emph{central}.
  The purpose of the grid is to provide a skeleton for another computational layer.
  If the image layer contains at least two period breakers, then at least one of the layers $X_{i,j}$ has a non-degenerate grid, on top of which we can force a computation.
  
  The final sub-layer is called the \emph{tile machine layer}, and it is used to run a computation inside each cell of the grid.
  These computations will be linked together in order to simulate a tiling in the grid.
  The purpose of the simulated tiling is to periodically send a signal to the verification machine layer, which in turn forces a row of $\#$-symbols that resets the computation of $M_{i,j}$.
  
  The computations of the tile machine layer are carried out by deterministic counter machines, which we embed into the grid shift.
  The simulation of computation will take place in the upper left half of each grid cell (copies of the solid white tile of Figure~\ref{fig:GridAlph}).
  For this, we take a deterministic $k$-counter machine $M = (k, Q, \delta, q_0, q_f)$ and modify the SFT $X_M$ defined in Section~\ref{sec:Machines}.
  We join the symbols $\#_L$ and $\#_R$ into one symbol, $\#$, which is overlaid on the tiles of the grid layer that are not solid white.
  On the solid white tiles we overlay other tiles of $S_M$.
  The bottom-most white tile initializes the computation in state $q_0$ with all internal counters set to $0$ (the input counters can have any value that fits in the grid cell).
  Once the machine enters the final state $q_f$, the zig-zag head stops and stays at the left end of the computation cone, similarly to the simulation of $M_{i,j}$.
  We require that the top left corner of each grid cell contains a zig-zag head in the state $q_f$.

  Note that the behavior of each simulated copy of $M$ is completely independent from the others.
  We now enhance the construction by allowing the machines to communicate with their neighbors.
  This is reminiscent of the fixed point construction of \cite{DuRoSh12}, where a tiling simulates a grid of Turing machines, which in turn collaborate to simulate another tileset.
  Each copy of $M$ has four read-only input counters: $\north{C}$, $\south{C}$, $\east{C}$ and $\west{C}$, called the \emph{north, south, east and west counters}.
  We would like to enforce by local rules that the north counter of each copy of $M$ is equal to the south counter of the copy to its north, and analogously for the east and west counters.
  In order to implement these constraints, we introduce new signals into the grid cells.
  We define new tile sets $\hsignal{S}$ and $\vsignal{S}$ in Figure~\ref{fig:SigAlph}.
  To each interior tile of a grid cell we overlay an element of the product alphabet $\signal{S} = \hsignal{S}^2 \times \vsignal{S}$, called the \emph{east, west and south signal layers}.
  In the interior of each grid cell, for each track of $\signal{S}$, the lines and colors of adjacent tiles must match, and the horizontal or vertical line must be present (this is enforced by requiring the corners of the grid cell to be adjacent to certain colors on each track).
  Also, the leftmost tile on the horizontal line must be the bottom left tile in Figure~\ref{fig:SigAlph}a.
  In each grid cell, the line on the west signal track must be at the same height as the line on the east signal track of the eastern neighboring grid cell.
  The east and west counters of each copy of $M$ must be positioned where the corresponding diagonal signals of $\signal{S}$ hit the top edge of the grid cell.
  The vertical line of the south signal layer must be positioned on the south counter, and the south signal must hit the southern border of the grid cell at the same position as the north signal of the southern neighboring grid cell.
  In this way, the desired couplings between the input counters of neighboring grid cells are achieved.

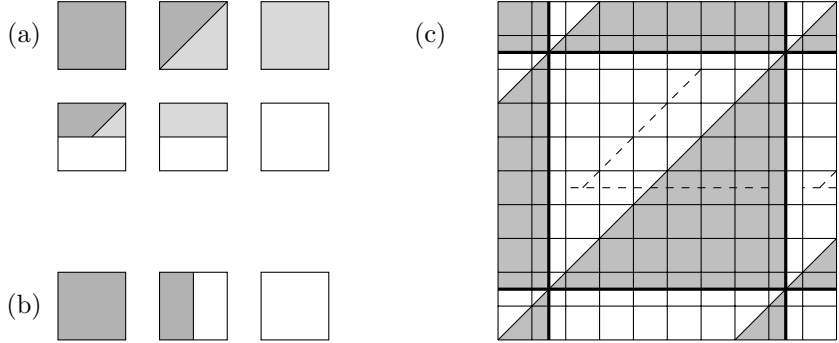
\begin{figure}[htp]
\begin{center}
\begin{tikzpicture}[scale=.9]

  \node at (0,4) {(a)};
  \node at (0,0) {(b)};
  \node at (6,4) {(c)};

  \begin{scope}[yscale=-1,yshift=-4.5cm,xshift=0.5cm]

  
  \begin{scope}[shift={(0,0)}]
  \draw[fill=black!30] (0,0) rectangle (1,1);
  \end{scope}
  
  \begin{scope}[shift={(1.5,0)}]
  \fill[black!30] (1,0) -- (0,1) -- (0,0);
  \fill[black!15] (1,0) -- (0,1) -- (1,1);
  \draw (1,0) -- (0,1);
  \draw (0,0) rectangle (1,1);
  \end{scope}
  
  \begin{scope}[shift={(3,0)}]
  \draw[fill=black!15] (0,0) rectangle (1,1);
  \end{scope}

  
  \begin{scope}[shift={(0,1.5)}]
  \fill[black!30] (0,0) -- (0,0.5) -- (0.5,0.5) -- (1,0);
  \fill[black!15] (0.5,0.5) -- (1,0.5) -- (1,0);
  \draw (0,0.5) -- (1,0.5);
  \draw (0.5,0.5) -- (1,0);
  \draw (0,0) rectangle (1,1);
  \end{scope}

  \begin{scope}[shift={(1.5,1.5)}]
  \fill[black!15] (0,0) rectangle (1,0.5);
  \draw (0,0.5) -- (1,0.5);
  \draw (0,0) rectangle (1,1);
  \end{scope}
  
  \begin{scope}[shift={(3,1.5)}]
  \draw (0,0) rectangle (1,1);
  \end{scope}
  
  \end{scope}

  \begin{scope}[xshift=0.5cm,yshift=-0.5cm]
  
  \begin{scope}[shift={(0,0)}]
  \draw[fill=black!30] (0,0) rectangle (1,1);
  \end{scope}
  
  \begin{scope}[shift={(1.5,0)}]
  \fill[black!30] (0,0) rectangle (0.5,1);
  \draw (0.5,0) -- (0.5,1);
  \draw (0,0) rectangle (1,1);
  \end{scope}
  
  \begin{scope}[shift={(3,0)}]
  \draw (0,0) rectangle (1,1);
  \end{scope}  
  
  \end{scope}

  
  \begin{scope}[shift={(7,-0.5)}]
  \fill [black!25] (0,0) rectangle (5,5);
  \fill [white] (0.75,0.75) -- (0,0) -- (0,0.75) -- (0.75,0.75) -- (0.75,0) -- (3.5,0) -- (4.25,0.75) -- (0.75,0.75) -- (4.25,4.25) -- (4.25,5) -- (5,5) -- (4.25,4.25) -- (4.25,0.75) -- (4.25,0) -- (5,0) -- (5,0.75) -- (4.25,0.75) -- (5,1.5) -- (5,4.25) -- (4.25,4.25) -- (0.75,4.25) -- (0,3.5) -- (0,4.25) -- (0.75,4.25) -- (0.75,5) -- (1.5,5) -- (0.75,4.25);
  \draw (0,0) -- (5,5);
  \draw (0,3.5) -- (1.5,5);
  \draw (3.5,0) -- (5,1.5);
  \draw [very thick] (0.75,0) -- (0.75,5);
  \draw [very thick] (4.25,0) -- (4.25,5);
  \draw [very thick] (0,0.75) -- (5,0.75);
  \draw [very thick] (0,4.25) -- (5,4.25);
  \draw [dashed] (5,2.25) -- (4.5,2.25);
  \draw [dashed] (4,2.25) -- (1,2.25);
  \draw [dashed] (1.25,2.25) -- (3,4);
  \draw [dashed] (4.75,2.25) -- (5,2.5);
  \draw [step=0.5] (0,0) grid (5,5);
  \end{scope}

\end{tikzpicture}
\end{center}
\caption{(a) The alphabet $\hsignal{S}$, (b) the alphabet $\vsignal{S}$, and (c) an illustration of the coupling between the west and east signals of neighboring grid cells. Counters and other signals are not shown.}
\label{fig:SigAlph}
\end{figure}

  Each simulated copy of $M$ executes the same program, which features a computable sequence of tile sets $(T_n)_{n \in \N}$.
  The edge colors of each $T_n$ are represented as integers, and the sets of colors of different tile sets are disjoint.
  Each machine checks that its north, south, east and west counters form a valid tile in one of the sets $T_n$, and if this is the case, enters the final state $q_f$.
  We say that the machine and the associated grid cell \emph{simulate} said tile.
  If the quadruple of numbers is not a valid tile, the machine halts and produces a tiling error.
  In this way, each grid cell acts as a tile of some $T_n$ (for the same $n$, since the colors are disjoint), and an infinite grid simulates a valid tiling over this set.
  Note that the number of computation steps generally depends on the values of the input counters, even if they represent tiles of the same set $T_n$.
  Since the simulated copies of $M$ are not required to halt at the same time and can wait in the state $q_f$ until the computation cone terminates, this is not an issue as long as the grid cells are large enough.
  
  We now describe the tile sets $T_n$ in more detail.
  Let $g : \N \to \N$ be a computable function depending on $Y$, which we fix later.
  The tiles of $T_n$ are similar to those of the grid shift in Example~\ref{ex:Grid}.
  In addition to the markings in Figure~\ref{fig:GridAlph}, on each tile not containing a horizontal or vertical line is superimposed a tile of an auxiliary tile set $R_n$.
  The tile sets $R_n$ have the following properties.
  \begin{itemize}
  \item Each $R_n$ is a north-west deterministic tile set.
  \item The tile set $R_n$ can tile a square region of the shape $[0, g(n)-1]^2$ so that the west border is colored uniformly with a color $W_n$, and the north border is colored uniformly with $N_n$.
  \item The tile set $R_n$ does not tile the infinite plane.
  \item The tile sets $R_n$ and the colors $W_n$ and $N_n$ are computable from $n$ by a counter machine in time $\exp(O(n^q))$ for some constant $q \in \N$, independently of the choice of $g$.
  \end{itemize}
  Such tile sets can be constructed as variants of the Robinson tiles of \cite{Ro71}, as was done in \cite{Ka92}.
  Note the requirement on the time complexity, which can be satisfied as follows.
  The construction relies on embedding a Turing machine computation in the Robinson tiling.
  If the machine halts after $m$ steps when initialized on an empty tape, the tile set can tile a square of shape $[0, m-1]^2$, but does not tile the infinite place.
  We can implement the tile sets $R_n$ using an embedded Turing machine that writes the number $n$ to the tape, counts to $g(n)$ and halts.
  Such an implementation can easily be realized in polynomial time by a Turing machine, which a counter machine can simulate with exponential slowdown.
  
  We require that the tiles of $R_n$ that are superimposed on tiles directly to the east of a vertical grid line of $T_n$ have west color $W_n$, and those that lie directly to the south of a horizontal grid line have north color $N_n$.
  The tiles of $R_n$ that lie to the north or west of the grid lines have no additional restrictions on their edge colors.
  The idea is that the tileset forces the occurrence of grid cells which are filled by tiles drawn from $R_n$, and hence have width and height at most $g(n)$.
  Because the tile set $R_n$ is north-west deterministic, each simulated grid cell has identical contents.
  Thus the SFT defined by $T_n$ is countable.
  We will use the horizontal rows of these simulated tilings to control the $\#$-rows of the verification machine layer.
  For this, we say that a grid cell of the grid layer is part of a \emph{reset row}, if the copy of $M$ running inside it simulates a tile of $T_n$ with a horizontal line.
  
  We relate the verification machine layer of $X_{i,j}$ to the grid and tile machine layers in the following way.
  The symbol $\#$ of the verification machine layer can only be paired with horizontal lines of the grid.
  If the zig-zag head of $M_{i,j}$ is in the final state $q_f$ and the next row of the grid layer contains a horizontal row, then the next row of the verification machine layer must be filled with $\#$-symbols.
  Moreover, on the bottom edge of those grid cells that are part of reset rows (whose copies of $M$ simulate the horizontal lines of $T_n$), the verification layer must have the symbol $\#$.
  This can be enforced by giving the information about the symbol on the verification machine layer to the copies of $M$ when they are initialized, as part of their internal state.
  These conditions have the following effects.
  The simulation of the counter machine $M_{i,j}$ is reset at the bottom edge of each reset row, and before this reset, the machine must have finished its computation and entered the final state.
  Conversely, when the machine $M_{i,j}$ has entered its final state, the next horizontal grid line must reset the computation, and the corresponding grid cells must form a reset row.
  Hence the computations of $M_{i,j}$ are synchronized with the simulated horizontal lines of $T_n$.
  Since $M_{i,j}$ is deterministic, it cannot change its behavior after a reset, so each configuration of $X_{i,j}$ containing a central column of grid cells is vertically periodic.
  
  We now describe the function $g$, which bounds the distance between two reset rows.
  For $m \in \N$, let $H(m)$ be the maximal amount of vertical space, measured in number of cells, needed for the simulation of the machine $M_{i,j}$ in a configuration where the distance between $v_i$ and $v_j$ (and thus the size of the grid cells) is $m$.
  In such a configuration, each simulated machine $M$ can perform at most $\log(m)$ computation steps before halting.
  For $n \in \N$, let $G(n)$ be the maximal number of steps required for $M$ to simulate a tile of $T_n$.
  The functions $G, H : \N \to \N$ are computable, and by our assumptions $G(n) = \exp(O(n^q))$, where the constants may depend on $g$.
  We may also assume that $G$ and $H$ are non-decreasing.
  Recall the function $g : \N \to \N$ in the definition of the tile sets $T_n$.
  We require
  \begin{equation}
  \label{eq:gBound}
    g(n) \geq H(\exp(G(n)))
  \end{equation}
  for all large enough $n \in \N$; this can be safely assumed since the right hand side is computable, and $H$ and $q$ are independent of the choice of $g$.
  For simplicity, we choose the same function $g$ for each choice of $i$ and $j$.
  Let $n_0 \in \N$ be such that~\eqref{eq:gBound} holds for all $n \geq n_0$, and denote $m_0 = \exp(G(n_0))$.
  
  Figure~\ref{fig:AllTogether} is a schematic diagram of a configuration of $X_{i,j}$.
  The squares are grid cells, each of which contains a simulated copy of $M$ (not shown).
  Dark gray squares simulate the horizontal, vertical and/or diagonal lines of $T_n$, while light gray and white squares simulate the uniformly colored tiles.
  The grid cells inside the large simulated grid cells also simulate the superimposed tiles of $R_n$.
  The thick horizontal lines are $\#$-symbols on the verification machine layer, and they lie on the bottom edges of the reset rows.
  The thick vertical line is the anchor of the verification machine layer, and the thick diagonal lines are the borders of the computation cone.
  The dotted line is the zig-zag head.
  Counters are not shown.
  The pattern repeats periodically to the north and south, and to the east and west only the simulated grid and the thick vertical lines are visible.
  
  \begin{figure}[htp]
  \begin{center}
  \begin{tikzpicture}

  \pgfmathsetmacro{\st}{1/3}
  \pgfmathsetmacro{\ht}{4}
  \pgfmathsetmacro{\sf}{5}
  
  \clip (-5-0.01,-0.01) rectangle (5+0.01,3*\ht+0.01);
  
  \foreach \y in {0,1,2}{
  \pgfmathsetmacro{\sh}{\y*\ht}
  \begin{scope}[yshift=\sh cm]
  
  \fill [black!35] (-5,0) rectangle (5,\st);
  \pgfmathsetmacro{\mx}{\ht-\st}
  \pgfmathsetmacro{\htt}{2*\ht}
  \pgfmathsetmacro{\stt}{2*\st}
  \foreach \xx in {-\htt,-\ht,0,\ht}{
    \fill [black!35] (\xx+\sf*\st,0) rectangle ++(\st,\ht);
    \foreach \x in {\st,\stt,...,\mx}{
      \fill [black!35] (\xx+\x+\sf*\st,\x) rectangle ++(\st,\st);
      \fill [black!15] (\xx+\x+\st+\sf*\st,\x) rectangle (\xx+\ht+\sf*\st,\x+\st);
    }
  }
  \draw [step=\st] (-5,0) grid (5,\ht);
  
  \coordinate (A) at (0,3+5/6);
  \coordinate (B) at (4,1+5/6);
  \coordinate (O) at (0,0);
  \coordinate (R) at (\ht,\ht);
  \coordinate (L) at (-\ht,\ht);
  
  \draw [ultra thick] (-5,0) -- (5,0);
  \draw [ultra thick] (0,0) -- (0,\ht);
  \draw [name path=PL,ultra thick] (O) -- (L);
  \draw [name path=PR,ultra thick] (O) -- (R);
  \path [name path=AB] (A) -- (B);
  
  \draw [thick,densely dotted,name intersections={of=AB and PR,by={C}}] (A) -- (C);
  \path [name path=CD] (C) -- ++(-6,-3);
  \draw [thick,densely dotted,name intersections={of=PL and CD,by={D}}] (C) -- (D);
  \path [name path=DE] (D) -- ++(2,-1);
  \draw [thick,densely dotted,name intersections={of=PR and DE,by={E}}] (D) -- (E);
  \path [name path=EF] (E) -- ++(-2,-1);
  \draw [thick,densely dotted,name intersections={of=PL and EF,by={F}}] (E) -- (F);
  
  \end{scope}
  }
  
  \end{tikzpicture}
  \end{center}
  \caption{A schematic diagram of a configuration of $X_{i,j}$.}
  \label{fig:AllTogether}
  \end{figure}
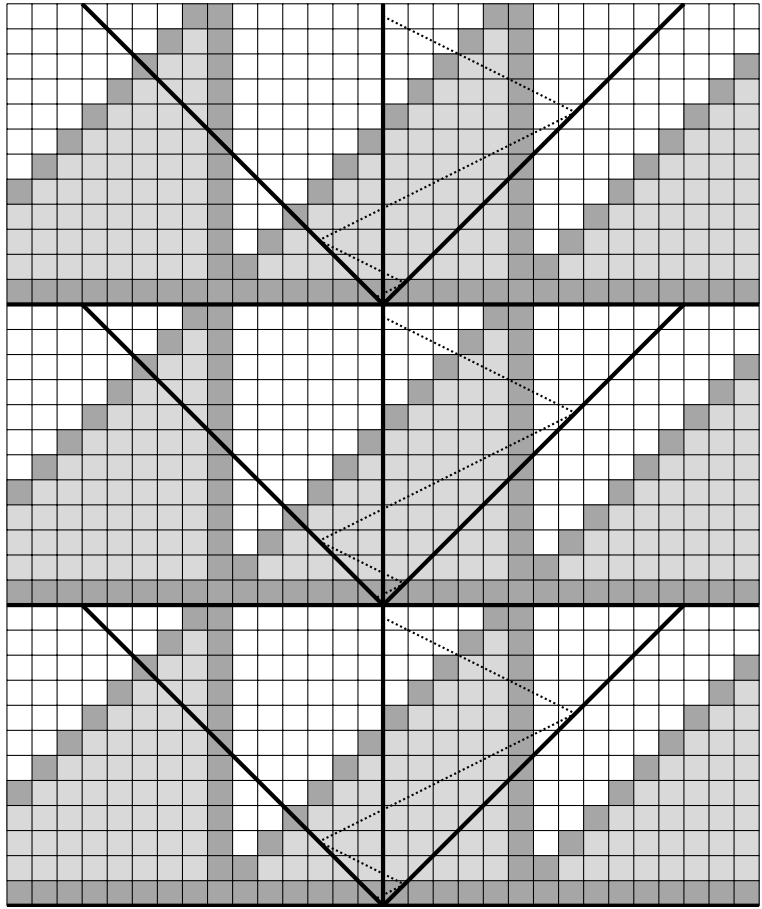
  
  We make one final modification to $X_{i,j}$.
  Suppose that the distance between the leftmost symbols of $v_i$ and $v_j$, or equivalently the size of the grid cells, is some $m \leq m_0$.
  In this case we allow the tile machine layer to be filled with a special blank symbol, so that no computation takes place and no tiling is simulated.
  We also require that for every $H(m)$ consecutive horizontal rows there is exactly one row of $\#$-symbols in the verification machine layer, which means that the machine $M_{i,j}$ is guaranteed to be present and have enough time to finish its computation before restarting periodically.
  These modifications can be realized by local rules, since they concern finitely many grid sizes and finitely many local patterns of the layers.
  The reason for these additions is that in too small grid cells, there is not enough space for the machines $M$ to simulate a suitable tile set $T_n$, but we must still force $M_{i,j}$ to perform its computation.
  This concludes the definition of the SFT $X$.
  
  It remains to prove two claims: that the SFT $X$ is countable, and that $\phi$ is a factor map from $X$ onto $Y$.
  We begin with the former claim.
  Let $z = {}^\infty u_{i-1} v_i u_i^{n_i} v_{i+1} \cdots v_j u_j^\infty \in Z$ be a configuration whose outermost period breakers are $v_i$ and $v_j$ for some $i < j$, and let $x \in X$ be a configuration whose image layer contains $z$.
  There are countably many choices for $z$, and we claim that there are countably many choices for $x$ once $z$ is fixed.
  For each $i \leq a < b \leq j$, the size of the grid cells and the positions of the vertical grid lines in $X_{a,b}$ are determined by the positions of the words $v_a$ and $v_b$.
  For other choices of $a < b$, the grid of $X_{a,b}$ has only one or zero vertical and horizontal lines, since at most one of the words $v_a$ and $v_b$ is present on the image layer.
  We call there layers $X_{a,b}$ \emph{degenerate}.
  For the positions of the horizontal lines of each grid (degenerate or not) we have a countable number of choices.
  
  The tile machine layer and verification machine layer of each $X_{a,b}$ work in unison to guarantee countability.
  If $X_{a,b}$ is degenerate, then there are only finitely many copies of the machines $M$ and $M_{a,b}$, so there are countably many choices for the computations.
  Due to our method of implementing the simulations, the set of degenerate configurations where no computation takes place is likewise countable.
  Suppose then that $X_{a,b}$ is not degenerate, and that the position of the grid has been fixed.
  If the grid is sparse enough, then the tile machine layer is not blank.
  Each grid cell hosts a computation of $M$, which simulates a tile of $T_n$, for some $n \in \N$ that all the machines agree on.
  The number of valid tilings of $T_n$ is countable, and each of them contains a grid of side length at most $g(n)$.
  Each horizontal line of $T_n$ forms a reset row that causes a new simulation of $M_{a,b}$ to be initialized at the lower left corner of the grid cell on the central column.
  Since $M_{a,b}$ is deterministic, each simulated copy performs the same computation.
  If the grid is dense, then the tile machine layer is blank, and $M_{a,b}$ is forced to perform its computation periodically due to the additional rules we imposed in this case.
  In any case, $x$ is completely determined by the positions of the $\#$-rows on the verification machine layer, for which we have countably many choices.
  Hence there are countably many choices for $x$.
  
  Finally, if the image layer contains a degenerate configuration of the type ${}^\infty u_{i-1} v_i u_i^\infty$ or ${}^\infty u_i^\infty$, then each layer $X_{a,b}$ is also degenerate, and has countably many choices for its contents as above.
  We have shown that $X$ is a countable SFT.
  
  Next, we claim that $\phi(x) \in Y$ for each $x \in X$.
  Suppose that $x$ has an image layer whose rows have the form $z = {}^\infty u_{i-1} v_i u_i^{n_i} v_{i+1} \cdots v_j u_j^\infty \in Z$.
  If the grid cells of $X_{i,j}$ have width less than $m_0$, then the tile machine layer is blank, and the verification machine layer periodically contains a computaion of $M_{i,j}$.
  If the grid cells have width at least $m_0$, then the tile machine layer is not blank.
  In particular, the copies of $M$ on the tile machine layer simulate a tiling of some $T_n$, which causes reset rows to occur periodically, and each of those initializes a simulated computation of $M_{i,j}$ on the verification machine layer.
  In both cases, $M_{i,j}$ performs a successful computation, so we have two possibilities: either $z \in P(Y)$, in which case we are done, or $z \notin P(Y)$, in which case the machine finds a number $p \in \N$ such that the word $w' = u_{i-1}^p v_i u_i^{n_i} v_{i+1} \cdots v_j u_j^p$ does not occur in $P(Y)$.
  After this, it scans $|u_{i-1}|^p$ cells to the left of $v_i$ and $|u_j|^p$ steps to the right of $v_j$, finds the word $w'$, and halts with a tiling error.
  This is a contradiction, so we must have $z \in P(Y)$.
  We have shown that $\phi(x) \in Y$.
  
  We claim that $\phi : X \to Y$ is surjective.
  Let $y \in Y$ be arbitrary, and let $z = {}^\infty u_{i-1} v_i u_i^{n_i} v_{i+1} \cdots v_j u_j^\infty \in P(Y) \subset Z$ be its central row.
  We construct a configuration $x \in X$ that has $y$ on its image layer.
  Let $1 \leq a < b \leq k$ be arbitrary.
  If $a < i$ or $j < b$, then the layer $X_{a,b}$ of $x$ can be degenerate: the grid layer contains a single vertical line on the leftmost symbol of $v_a$ or $v_b$, if either is present, or no vertical line at all.
  It contains no horizontal line.
  The tile machine layer and the verification machine layer contain no machines, just infinite regions tiled with copies of a single tile.
  
  Suppose then that $i \leq a < b \leq j$.
  On the grid layer of $X_{a,b}$, we have two vertical lines on the leftmost symbols of $v_a$ and $v_b$ with distance $m > 0$, and no other vertical lines between them.
  If $m \leq m_0$, then we fill the tile machine layer with blank symbols, and place rows of $\#$-symbols at regular intervals of length $H(m)$ on the verification machine layers.  
  Suppose now $m > m_0 = \exp(G(n_0))$, and let $n \geq n_0$ be the largest natural number with $m > \exp(G(n))$.
  On the tile machine layer of $X_{a,b}$, we can hence simulate a copy of $M$ in each grid cell, and they have enough space to simulate tiles of the set $T_n$.
  Furthermore, we have $m \leq \exp(G(n-1))$, and from equation~\eqref{eq:gBound} we obtain $g(n) \geq H(\exp(G(n))) \geq H(\exp(G(n-1))) \geq H(m)$ since $G$ and $H$ are non-decreasing.
  Thus the distance between two reset rows can be at least $H(m)$ cells.
  
  We claim that in both cases, the spaces between these rows can be filled with valid computations of $M_{a,b}$.
  If ${}^\infty u_{a-1} v_a u_a^{n_a} \cdots v_b u_b^\infty \in P(Y)$, then the machine scans this word and enters the final state before $H(m)$ steps.
  Otherwise, there exists $p \in \N$ such that $w = u_{a-1}^p v_a u_a^{n_a} \cdots v_b u_b^p$ does not occur in $P(Y)$, and hence does not occur in $z$.
  The machine computes this $p$, reads the word of $z$ containing $v_a u_a^{n_a} \cdots v_b$, verifies that it is not equal to $w$, and enters the final state before $H(m)$ steps.
  
  We have shown that these layers form a valid configuration of $X_{a,b}$ consistent with the image layer $y$.
  Thus $\phi$ is surjective, and hence $X$ is a countable SFT cover of $Y$.
\end{proof}

\section{Conclusions}

We have provided simple examples of countable sofic shifts that have no countable SFT covers.
All of them are vertically constant, and rely on Lemma~\ref{lem:UniformlyRecurrent} for the proof of this property.
It seems that periodicity is an essential component of our proof technique and cannot be dispensed with easily.
We have no examples of two-dimensional countable sofic shifts without countable SFT covers where this property does not follow from the existence of a periodic component of some kind.
In \cite{We17} it was shown that all effectively closed subsystems of the \emph{distinct square shift} are sofic.
The distinct square shift is the two-dimensional binary shift space $X$ whose configurations consist of square patterns of $1$s with disjoint borders and distinct sizes in a background of $0$s.
It seems likely that some sufficiently complex countable subshifts of $X$ would have no countable SFT cover, but we are currently unable to prove this.

On the other hand, we do not know whether the countable cover conditions are sufficient for the existence of a countable SFT cover in all vertically constant shift spaces.
Let us explicitly state this as an open problem.

\begin{question}
  \label{q:TheQuestion}
  Is it true that a vertically constant shift space is a countably covered sofic shift if and only if it satisfies the countable cover conditions?
\end{question}

Theorem~\ref{thm:Converse} provides a positive answer for shift spaces whose rows come from a one-dimensional countable sofic shift, and we have provided examples outside of this class.
It is not clear how to extend our construction to a larger class of shift spaces, since it depends on the simple structure of one-dimensional countable sofic shifts, namely the bounded number of period breakers.

Our proof of Lemma~\ref{lem:UniformlyRecurrent} uses Zorn's lemma, and thus relies on some version of the Axiom of Choice.
We are not certain whether a different approach could avoid its use in the special case of countable SFTs.

Finally, as stated in the Introduction, this research was motivated by the equal entropy SFT cover problem for multidimensional sofic shifts, and in particular the question whether each zero entropy sofic shift has a zero entropy cover.
As our approach relies on periodicity and the topological properties of countable shift spaces, it cannot be directly applied to all zero entropy sofic shifts.
Furthermore, the constructions of \cite{DuRoSh12,AuSa13} give zero entropy SFT covers to all vertically constant sofic shifts, so counterexamples cannot be found inside this class.

\section*{Acknowledgements}

The author is thankful to Andrei Romashchenko, Bruno Durand and Pierre Guillon for discussions and comments on this article, and to Ville Salo for Example~\ref{ex:Pi01Needed}.

\bibliographystyle{plain}
\bibliography{StripesBib}

\end{document}